\documentclass[11pt,letterpaper]{article}

\usepackage[T1]{fontenc}
\usepackage[latin1]{inputenc}
\usepackage{amsmath,amsthm,amstext,amsfonts,amssymb,amscd,amsxtra}
\usepackage{bm}
\usepackage{enumitem}
\usepackage{comment}
\usepackage{eucal}
\usepackage{mathrsfs,txfonts}
\usepackage{stmaryrd}
\usepackage[all]{xy}
\usepackage{url}
\usepackage{xcolor}
\usepackage{fancyhdr}

\allowdisplaybreaks

\newcounter{step}
\newcommand{\Proofstep}{\par\refstepcounter{step}Step~\thestep.\space\ignorespaces}

\theoremstyle{plain}
\newtheorem{theorem}{Theorem}[section]
\newtheorem*{theoremA}{Theorem~A}
\newtheorem*{theoremB}{Theorem~B}
\newtheorem*{corollaryC}{Corollary~C}
\newtheorem*{corollaryD}{Corollary~D}

\newtheorem{lemma}[theorem]{Lemma}
\newtheorem{proposition}[theorem]{Proposition}
\newtheorem{proposition-definition}[theorem]{Proposition-Definition}
\newtheorem{corollary}[theorem]{Corollary}
\newtheorem{claim}[theorem]{Claim}

\theoremstyle{definition}
\newtheorem{definition}[theorem]{Definition}

\theoremstyle{remark}
\newtheorem{remark}[theorem]{Remark}

\renewcommand{\labelenumi}{(\arabic{enumi})}

\newcommand{\ZZ}{\mathbb{Z}}
\newcommand{\QQ}{\mathbb{Q}}
\newcommand{\RR}{\mathbb{R}}
\newcommand{\CC}{\mathbb{C}}
\newcommand{\FF}{\mathbb{F}}
\newcommand{\KK}{\mathbb{K}}
\newcommand{\PP}{\mathbb{P}}
\newcommand{\OO}{\mathcal{O}}

\makeatletter
 
 \@addtoreset{equation}{section}
\makeatother

\makeatletter
 \newcommand{\esssup}{\mathop{\operator@font ess.sup}\displaylimits}
 \newcommand{\essinf}{\mathop{\operator@font ess.inf}\displaylimits}
\makeatother

\makeatletter
 \def\BIG#1{%
  {\hbox{$\left#1\vbox to20.5\p@{}\right.\n@space$}}}

\makeatother

\makeatletter 
 \let\@@pmod\pmod 
 \DeclareRobustCommand{\pmod}{\@ifstar\@pmods\@@pmod} 
 \def\@pmods#1{\mkern4mu({\operator@font mod}\mkern 6mu#1)} 
\makeatother

\renewcommand{\div}{\mathop{\mathrm{div}}\nolimits}
\newcommand{\Image}{\mathop{\mathrm{Image}}\nolimits}
\renewcommand{\leq}{\leqslant}
\renewcommand{\geq}{\geqslant}

\DeclareMathOperator{\Sym}{Sym}
\DeclareMathOperator{\Spec}{Spec}
\DeclareMathOperator{\Proj}{Proj}
\DeclareMathOperator{\Supp}{Supp}
\DeclareMathOperator{\Rat}{Rat}
\DeclareMathOperator{\Div}{CDiv}

\DeclareMathOperator{\aTheta}{\widehat{\Theta}}

\DeclareMathOperator{\aDiv}{\widehat{CDiv}}
\DeclareMathOperator{\aCDiv}{\widehat{CDiv}}
\DeclareMathOperator{\aInt}{\widehat{Int}}
\DeclareMathOperator{\aDDiv}{\widehat{\mathbb{CD}iv}}

\DeclareMathOperator{\Conv}{Conv}
\DeclareMathOperator{\adeg}{\widehat{deg}}
\DeclareMathOperator{\ord}{ord}
\DeclareMathOperator{\vol}{vol}
\DeclareMathOperator{\avol}{\widehat{vol}}
\DeclareMathOperator{\rk}{rk}

\newcommand{\ah}{\widehat{\ell}^{\ast}}
\newcommand{\ahss}{\widehat{\ell}^{\rm ss}}

\newcommand{\quot}{{\rm quot}}

\newcommand{\sbullet}{{\scriptscriptstyle\bullet}}

\def\avolq#1{\avol_{#1}}

\title{Differentiability of the arithmetic volume function along the base conditions}
\author{Hideaki Ikoma\footnote
{
Department of Education, Shitennoji University,
Habikino Osaka 583-8501 Japan.
E-mail: h-ikoma@shitennoji.ac.jp.\
}
}
\date{\today}
\begin{document}
%\selectlanguage{english}

\maketitle

\begin{abstract}
In this paper, we show that the arithmetic volume function defined on the space of pairs of adelic $\RR$-Cartier divisors and base conditions is differentiable at a big pair, and that its derivative is given by an arithmetic restricted positive intersection number defined for the pair.
\end{abstract}

\tableofcontents

%%%Introduction
\section{Introduction}\label{sec:Intro}

Let $K$ be a number field, let $M_K^{\rm f}$ denote the set of finite places of $K$, and set $M_K\coloneqq M_K^{\rm f}\cup\{\infty\}$.
Let $X$ be a normal, projective, and geometrically connected variety defined over $K$, and let $\Rat(X)$ denote the field of rational functions on $X$.
An \emph{adelic $\RR$-Cartier divisor} $\overline{D}$ on $X$ is a couple $\left(D,\pmb{g}^{\overline{D}}\right)$ consisting of an $\RR$-Cartier divisor $D$ on $X$ and a family of $D$-Green functions
\[
\pmb{g}^{\overline{D}}=\sum_{v\in M_K}g_v^{\overline{D}}[v]
\]
satisfying the \emph{adelic condition} (see \cite{MoriwakiAdelic} for detail).
To each pair $\left(\overline{D};E\right)$ of an adelic $\RR$-Cartier divisor $\overline{D}$ on $X$ and an $\RR$-Cartier divisor $E$ on $X$, we assign the finite set of all \emph{strictly small} sections of $\overline{D}$ vanishing along the positive part of $E$: namely
\[
 \widehat{\Gamma}^{\rm ss}\left(\overline{D};E\right)\coloneqq\left\{\phi\in\Rat(X)^{\times}\,\colon\,\overline{D}+\widehat{(\phi)}>0\,\text{and}\, D+(\phi)\geq E\right\}\cup\{0\}
\]
(see Notation and terminology~\ref{NC:Vol}).
The \emph{arithmetic volume} of the pair $\left(\overline{D};E\right)$ is then defined as
\[
 \avol\left(\overline{D};E\right)\coloneqq\limsup_{\substack{m\in\ZZ_{\geq 1}, \\ m\to\infty}}\frac{\log\left(\#\widehat{\Gamma}^{\rm ss}\left(m\overline{D};mE\right)\right)}{m^{\dim X+1}/(\dim X+1)!}
\]
(see \cite{IkomaDiff1} for detail).

As is well-known, differentiability of the arithmetic volume function has essential importance in Arakelov geometry and many applications in problems on rational points (see for example \cite{SilvermanADS}).
Due to Yuan's arithmetic Siu inequality \cite{Yuan07}, we know that the arithmetic volume function above is G\^ateaux differentiable along directions defined by adelic $\RR$-Cartier divisors, and that the derivatives are given by \emph{arithmetic positive intersection numbers} (see \cite{Chen11,IkomaCon,IkomaDiff1}).
Another analytic approach to differentiability is proposed by Berman and Boucksom \cite{Berman_Boucksom}.
In this paper, we study G\^ateaux differentiability of the arithmetic volume function along directions defined by $\RR$-Cartier divisors, and show that the derivatives are given by \emph{arithmetic restricted positive intersection numbers}, which gives an answer to a question raised by Moriwaki and Chen.

Let $Y$ be a prime Cartier divisor on $X$.
A pair $\left(\overline{D};E\right)$ is said to be \emph{$Y$-big} if there is a \emph{weakly ample} (or \emph{w-ample} for short) adelic $\RR$-Cartier divisor $\overline{A}$ on $X$ (see Notation and terminology~\ref{NC:positive} for definition of w-ampleness) such that the pair $\left(\overline{D}-\overline{A};E\right)$ is strictly effective and the support of $D-A-E$ does not contain $Y$ as a component.
In particular, the $Y$-bigness of $\left(\overline{D};E\right)$ implies that $\ord_Y(E)\geq 0$.
The main purpose of this paper is to establish the following:
\medskip

\begin{theoremA}[see section~\ref{subsec:Yuan}]
Let $X$ be a normal projective variety over a number field, let $Y$ be a prime Cartier divisor on $X$, and let $\left(\overline{D};E\right)$ be a $Y$-big pair on $X$.
If $\ord_Y(E)>0$, then the function $r\mapsto\avol\left(\overline{D};E+rY\right)$ is two-sided differentiable at $r=0$ and
\[
 \lim_{r\to 0}\frac{\avol\left(\overline{D};E\right)-\avol\left(\overline{D};E+rY\right)}{r}=(\dim X+1)\left.\left\langle\left(\overline{D};E\right)^{\cdot\dim X}\right\rangle\right|_Y.
\]
\end{theoremA}
\medskip

The right-hand side of Theorem~A denotes the \emph{arithmetic restricted positive intersection number of $\left(\overline{D};E\right)$ along $Y$}, which we define as follows:
We refer to a couple $\left(\pi\colon X'\to X,\overline{M}\right)$ consisting of a modification $\pi\colon X'\to X$ and a nef and $\pi_*^{-1}(Y)$-big adelic $\RR$-Cartier divisor $\overline{M}$ on $X'$ such that $\left(\pi^*\overline{D}-\overline{M};E\right)$ is $\pi_*^{-1}(Y)$-pseudo-effective as a \emph{$Y$-approximation} of $\left(\overline{D};E\right)$, and set
\[
\left.\left\langle\left(\overline{D};E\right)^{\cdot\dim X}\right\rangle\right|_Y\coloneqq\sup_{\left(\pi,\overline{M}\right)}\left\{\adeg\left(\left(\left.\overline{M}\right|_{\pi_*^{-1}(Y)}\right)^{\cdot\dim X}\right)\right\},
\]
where the supremum is taken over all $Y$-approximations $\left(\pi,\overline{M}\right)$ of $\left(\overline{D};E\right)$.
\medskip

Under our definition of local positivity of pairs, the cone of $Y$-big pairs admits the boundary defined as
\[
\left\{\left(\overline{D};E\right)\,\colon\,\text{$\displaystyle{\left(\overline{D};E\right)}$ is $Y$-big and $\ord_Y(E)=0$}\right\}.
\]
Hence continuity of the arithmetic restricted positive intersection numbers at the boundary does not directly follows from their concavity property.
We partially solve this problem as follows (see Notation and terminology~\ref{NC:InnProd}):
\medskip

\begin{theoremB}[see section~\ref{subsec:Cont_at_the_boundary}]
Let $X$ be a normal projective variety over a number field, let $Y$ be a prime Cartier divisor on $X$, and let $\overline{A}$ be a nef and $Y$-big $\RR$-Cartier divisor on $X$.
Let $\overline{\bm{D}}\coloneqq\left(\overline{D}_1,\dots,\overline{D}_m\right)$ be a family of adelic $\RR$-Cartier divisors on $X$, and let $\bm{E}\coloneqq\left(E_1,\dots,E_n\right)$ be a family of $\RR$-Cartier divisors on $X$ such that $\ord_Y(E_j)=0$ for every $j$.
Then
\[
 \lim_{\substack{\bm{t},\bm{u}\to 0 \\ r\downarrow 0}}\left.\left\langle\left(\overline{A}+\bm{t}\cdot\overline{\bm{D}};rY+\bm{u}\cdot\bm{E}\right)^{\cdot\dim X}\right\rangle\right|_Y 
=\left.\left\langle\overline{A}^{\cdot\dim X}\right\rangle\right|_Y.
\]
\end{theoremB}
\medskip

Immediate consequences of our theorems are Corollaries~C and D below.
\medskip

\begin{corollaryC}
Let $X$ be a normal projective variety over a number field, and let $Y$ be a prime Cartier divisor on $X$.
If $\left(\overline{D};rY\right)$ is $Y$-big, then
\[
\avol\left(\overline{D};rY\right)=\left\langle\left(\overline{D};rY\right)^{\cdot \dim X}\right\rangle\cdot \overline{D}-r\left.\left\langle\left(\overline{D};rY\right)^{\cdot \dim X}\right\rangle\right|_Y.
\]
\end{corollaryC}
\medskip

\begin{corollaryD}
Let $X$ be a normal, projective variety over a number field, let $Y$ be a prime Cartier divisor on $X$, and let $\overline{A}$ be a nef and $Y$-big adelic $\RR$-Cartier divisor on $X$.
\begin{enumerate}
\item The function $r\mapsto\avol\left(\overline{A};rY\right)$ is one-sided differentiable at $r=0$ and
\[
 \lim_{r\downarrow 0}\frac{\avol\left(\overline{A}\right)-\avol\left(\overline{A};rY\right)}{r}=(\dim X+1)\left.\left\langle\overline{A}^{\cdot\dim X}\right\rangle\right|_Y.
\]
\item Let $\pmb{g}^{\overline{Y}}$ be an adelic $Y$-Green function, and consider the pair $\left(Y,\pmb{g}^{\overline{Y}};Y\right)$.
Then
\begin{align*}
\lim_{r\downarrow 0}\frac{\avol\left(\overline{A}+r\overline{Y};rY\right)-\avol\left(\overline{A}\right)}{r} &=(\dim X+1)\left(\left\langle\overline{A}^{\cdot \dim X}\right\rangle\cdot \overline{Y}-\left.\left\langle\overline{A}^{\cdot \dim X}\right\rangle\right|_Y\right) \\
&=(\dim X+1)\int_X\pmb{g}^{\overline{Y}}\,\left\langle\overline{A}^{\cdot \dim X}\right\rangle.
\end{align*}
\end{enumerate}
\end{corollaryD}
\medskip

In his paper \cite{Yuan09}, Yuan uses vertical flags
\[
\mathscr{F}_{\geq 1}\colon \mathscr{Y}=\mathscr{F}_1\supset\mathscr{F}_2\supset\dots\supset\mathscr{F}_{\dim\mathscr{Y}+1}
\]
on an arithmetic variety $\mathscr{Y}$ to construct convex bodies in Euclidean spaces whose Euclidean volumes approximate the arithmetic volume of a Hermitian line bundle on $\mathscr{Y}$ (see section~\ref{subsec:YuansEst}).
Later, Moriwaki \cite{MoriwakiEst} applies Yuan's techniques to the study of arithmetic restricted volumes of Hermitian line bundles.
After that, Yuan \cite{Yuan12} further constructs an \emph{arithmetic Newton--Okounkov body} whose Euclidean volume exactly gives the arithmetic volume of a given Hermitian line bundle.

The strategy to prove Theorem~A is then as follows:
First, in section~\ref{subsec:YuansEst}, we apply Yuan's techniques to the case of pairs and establish the arithmetic Fujita approximations for arithmetic restricted volumes of pairs, which ensures the identities between the arithmetic restricted volumes and the arithmetic restricted positive intersection numbers (see Proposition~\ref{prop: arithmetic Fujita approximation}).
Next, after observing a basic result on concave functions in section~\ref{subsec: concave}, we use Moriwaki's method \cite{MoriwakiCont} to give upper bounds for the derivatives of the arithmetic volume function in section~\ref{subsec: upper bound}.
In section~\ref{subsec: estimation of NO bodies}, we consider flags
\[
\mathscr{F}_{\sbullet}\colon \mathscr{X}\supset\mathscr{Y}=\mathscr{F}_1\supset\mathscr{F}_2\supset\dots\supset\mathscr{F}_{\dim\mathscr{X}}
\]
on $\mathscr{X}$ and construct the \emph{approximate} arithmetic Newton--Okounkov bodies for w-ample adelic Cartier divisors.
Lastly, by using these convex bodies, we will show lower bounds for the derivatives of the arithmetic volume function (see section~\ref{subsec:Yuan}).

\section*{Notation and terminology}\label{subsec:notation_and_terminology}

\begin{enumerate}
\renewcommand{\labelenumi}{\arabic{enumi}.}
\item\label{NC:norm}
The floor (respectively, ceiling) function is defined as
\begin{align*}
&\lfloor\alpha\rfloor\coloneqq\max\{n\in\ZZ\,\colon\,n\leq\alpha\} \\
\text{(respectively, }\quad &\lceil\alpha\rceil\coloneqq\min\{n\in\ZZ\,\colon\, n\geq \alpha\}\quad\text{)}
\end{align*}
for $\alpha\in\RR$.
For any $\bm{r}\coloneqq (r_1,\dots,r_l)\in\RR^l$, we set
\[
\lvert\bm{r}\rvert\coloneqq\left(\lvert r_1\rvert,\dots,\lvert r_l\rvert\right)\qquad\text{and}\qquad \|\bm{r}\|\coloneqq \sum_{i=1}^l|r_i|.
\]

\item\label{NC:InnProd}
Let $R$ be a ring, and let $M$ be an $R$-module.
The $R$-submodule of $M$ generated by a subset $\Gamma\subset M$ is denoted by $\langle\Gamma\rangle_R$.
Let $l\in\ZZ_{\geq 1}$, let $\bm{r}\coloneqq (r_1,\dots,r_l)\in R^l$ and let $\bm{m}\coloneqq (m_1,\dots,m_l)\in M^l$.
We use the dot-product notation as
\[
\bm{r}\cdot\bm{m}\coloneqq \sum_{i=1}^lr_im_i.
\]

\item\label{NC:flag} 
Let $X$ be a reduced, irreducible, and Noetherian scheme of finite Krull dimension.
We denote the field of rational functions on $X$ by $\Rat(X)$.
A \emph{flag on $X$} is a sequence of reduced, irreducible, and closed subschemes of $X$,
\[
F_{\sbullet}\colon X=F_0\supset F_1\supset F_2\supset\dots\supset F_{\dim X},
\]
such that each $F_i$ has codimension $i$ in $X$, such that $F_{\dim X}$ consists of a closed point $\xi$ of $X$, and such that each $F_{i+1}$ is locally principal in $F_i$ around $\xi$.

We define the \emph{valuation map $\bm{w}_{F_{\sbullet}}\colon \Rat(X)^{\times}\to\ZZ^{\dim X}$ attached to a flag $F_{\sbullet}$} as follows (see \cite[section 1.1]{Lazarsfeld_Mustata08}):
For each $i=1,\dots,\dim X$, we choose a local equation $f_i$ defining $F_i$ in $F_{i-1}$ around $\xi$.
Given a $\phi\in\Rat(X)^{\times}$, we set $\phi_1\coloneqq\phi$, and set
\[
 \phi_{i+1}\coloneqq\left.\left(f_i^{-\ord_{F_i}(\phi_i)}\cdot\phi_i\right)\right|_{F_i}
\]
for $i=1,\dots,\dim X-1$, inductively.
Then
\[
 \bm{w}_{F_{\sbullet}}(\phi)=(w_1(\phi),\dots,w_{\dim X}(\phi))\coloneqq\left(\ord_{F_1}(\phi_1),\dots,\ord_{F_{\dim X}}(\phi_{\dim X})\right),
\]
which does not depend on a specific choice of $f_1,\dots,f_{\dim X}$.

\item
Assume that $X$ is a normal.
Let $\KK$ denote either $\ZZ$, $\QQ$ or $\RR$.
The $\KK$-module of all $\KK$-Cartier divisors on $X$ is denoted by $\Div_{\KK}(X)$.
Given any $D\in\Div_{\KK}(X)$, we set
\begin{equation}
H^0(D)\coloneqq\left\{\phi\in\Rat(X)^{\times}\,\colon \,D+(\phi)\geq 0\right\}\cup\{0\}.
\end{equation}

\item\label{NC:Green}
Let $K$ denote a number field, let $M_K^{\rm f}$ denote the set of finite places of $K$, and let $M_K\coloneqq M_K^{\rm f}\cup\{\infty\}$.
For each $v\in M_K^{\rm f}$, $K_v$ denotes the $v$-adic completion of $K$.
Let $X$ be a normal, projective, and geometrically connected $K$-variety, let $X_v^{\rm an}$ denote the Berkovich analytic space associated to $X\times_{\Spec(K)}\Spec(K_v)$ for $v\in M_K^{\rm f}$, and let $X_{\infty}^{\rm an}$ denote the complex analytic space associated to $X\times_{\Spec(\QQ)}\Spec(\CC)$.
Let $D$ be an $\RR$-Cartier divisor on $X$.
The \emph{support of $D$} is defined as
\[
\Supp(D)\coloneqq\bigcup_{\ord_Z(D)\neq 0}Z,
\]
where the union is taken over all codimension-one subvarieties $Z$ of $X$ such that $\ord_Z(D)\neq 0$ (see \cite[Notation and terminology~2]{IkomaDiff1}).

For each $v\in M_K$, a \emph{$D$-Green function} on $X_v^{\rm an}$ is a function $g_v\colon (X\setminus\Supp(D))_v^{\rm an}\to\RR$ such that, for each $x_0\in X_v^{\rm an}$, the function
\[
g_v(x)+\log\left(|f|(x)\right)
\]
extends to a continuous function defined around $x_0$, where $f$ denotes a local equation defining $D$ around $x_0$.

Let $(\mathscr{X},\mathscr{D})$ be a normal and projective $O_K$-model of $(X,D)$.
For each $v\in M_K^{\rm f}$, $\widetilde{\mathscr{X}}_v$ denotes the fiber over $v$ and $r_v\colon X_v^{\rm an}\to\widetilde{\mathscr{X}}_v$ denotes the reduction map over $v$.
The \emph{$D$-Green function associated to $(\mathscr{X},\mathscr{D})$} is defined as
\begin{equation}
g_v^{(\mathscr{X},\mathscr{D})}(x)=-\log\left(|f|(x)\right),
\end{equation}
where $f$ is a local equation defining $\mathscr{D}$ around $r_v(x)$.

\item\label{NC:adelic}
Let $D$ be an $\RR$-Cartier divisor on $X$.
An \emph{adelic $D$-Green function} $\pmb{g}$ is a formal sum
\[
\pmb{g}\coloneqq\sum_{v\in M_K}g_v[v]
\]
having the following properties:
\begin{enumerate}
\item For each $v\in M_K$, $g_v$ is a $D$-Green function on $X_v^{\rm an}$ and $g_{\infty}^{\rm an}$ is invariant under the complex conjugation.
\item There exists a normal and projective $O_K$-model $(\mathscr{X},\mathscr{D})$ of $(X,D)$ such that $g_v=g_v^{(\mathscr{X},\mathscr{D})}$ for all but finitely many $v$.
\end{enumerate}
The $O_K$-model $(\mathscr{X},\mathscr{D})$ appearing in the property (b) above is called a \emph{model of definition for $\pmb{g}$}.
Let $\KK$ be either $\RR$, $\QQ$, or $\ZZ$.
We refer to a couple of a $\KK$-Cartier divisor $D$ on $X$ and an adelic $D$-Green function $\pmb{g}^{\overline{D}}$ as an \emph{adelic $\KK$-Cartier divisor $\overline{D}$} on $X$.
The $\KK$-module of all adelic $\KK$-Cartier divisors on $X$ is denoted by $\aDiv_{\KK}(X)$.

Given a nonzero rational function $\phi$,
\begin{equation}
\widehat{(\phi)}\coloneqq\left(\div(\phi),\sum_{v\in M_K}-\log\left(|\phi|\right)[v]\right)
\end{equation}
is an adelic Cartier divisor on $X$.

Let $\overline{D}$ be an adelic $\RR$-Cartier divisor on $X$, and let $\varphi$ be a continuous function on $X_{\infty}^{\rm an}$ that is invariant under the complex conjugation.
Then we denote
\begin{equation}\label{eqn:NC:adelic}
\overline{D}(\varphi)\coloneqq\overline{D}+(0,\varphi[\infty]).
\end{equation}

Let $\KK$ and $\KK'$ denote either $\RR$, $\QQ$, or $\ZZ$.
The module of all pairs of adelic $\KK$-Cartier divisors on $X$ and $\KK'$-Cartier divisors on $X$ is denoted by $\aDDiv_{\KK,\KK'}(X)$.

Let $\mathscr{X}$ be a normal and projective $O_K$-model of $X$, and let $\left(\overline{\mathscr{D}};\mathscr{E}\right)$ be a couple of an arithmetic $\RR$-Cartier divisor $\overline{\mathscr{D}}=\left(\mathscr{D},g^{\overline{\mathscr{D}}}\right)$ on $\mathscr{X}$ and a horizontal $\RR$-Cartier divisor $\mathscr{E}$ on $\mathscr{X}$.
We define the \emph{adelization} of $\left(\overline{\mathscr{D}};\mathscr{E}\right)$ as
\begin{equation}
\left(\overline{\mathscr{D}};\mathscr{E}\right)^{\rm ad}\coloneqq\left(\mathscr{D}|_X,\sum_{v\in M_K^{\rm f}}g_v^{(\mathscr{X},\mathscr{D})}[v]+g^{\overline{\mathscr{D}}}[\infty];\mathscr{E}|_X\right).
\end{equation}

\item\label{NC:Vol}
A pair $\left(\overline{D};E\right)\in\aDDiv_{\RR,\RR}(X)$ is said to be \emph{effective} if $D\geq \max\{0,E\}$ and $g_v^{\overline{D}}\geq 0$ on $X_v^{\rm an}$ for every $v\in M_K$.
We say that $\left(\overline{D};E\right)$ is \emph{strictly effective} if $\left(\overline{D};E\right)$ is effective and $\inf_{x\in X_{\infty}^{\rm an}}\left\{g_{\infty}^{\overline{D}}(x)\right\}>0$.
We denote $\left(\overline{D};E\right)\geq 0$ (respectively, $\left(\overline{D};E\right)>0$) if $\left(\overline{D};E\right)$ is effective (respectively, strictly effective).
We set
\begin{align}
&\widehat{\Gamma}^{\rm ss}\left(\overline{D};E\right)\coloneqq\left\{\phi\in\Rat(X)^{\times}\,\colon \,\left(\overline{D}+\widehat{(\phi)};E\right)>0\right\}\cup\{0\}, \\
&\widehat{\Gamma}^{\rm s}\left(\overline{D};E\right)\coloneqq\left\{\phi\in\Rat(X)^{\times}\,\colon \,\left(\overline{D}+\widehat{(\phi)};E\right)\geq 0\right\}\cup\{0\},
\end{align}
and define
\begin{equation}
\avol\left(\overline{D};E\right)\coloneqq\limsup_{\substack{m\in\ZZ_{\geq 1}, \\ m\to\infty}}\frac{\log\left(\#\widehat{\Gamma}^{\rm ss}\left(m\overline{D};mE\right)\right)}{m^{\dim X+1}/(\dim X+1)!}.
\end{equation}

\item\label{NC:positive}
The \emph{height} of an algebraic point $x\in X(\overline{K})$ with respect to $\overline{D}$ is defined as
\[
 h_{\overline{D}}(x)\coloneqq\frac{1}{[\kappa(x):K]}\left(\sum_{w\in M_{\kappa(x)}^{\rm f}}g_{w|_K}(x^w)+\sum_{\sigma\colon \kappa(x)\to\CC}g_{\infty}(x^{\sigma})\right),
\]
where $\kappa(x)$ denotes the field of definition for $x$, $x^w$ denotes the point on $X_v^{\rm an}$ corresponding to $(\kappa(x),w)$, and $x^{\sigma}$ denotes the point on $X_{\infty}^{\rm an}$ defined by an embedding $\sigma\colon \kappa(x)\to\CC$.

\begin{description}
\item[(nef)] We say that $\overline{A}\in\aDiv_{\RR}(X)$ is \emph{nef} if $A$ is nef, $g_v^{\overline{A}}$ is semipositive for every $v\in M_K$ (see \cite[section 4.4]{MoriwakiAdelic}), and
\[
 \inf_{x\in X(\overline{K})}\left\{h_{\overline{A}}(x)\right\}\geq 0.
\]
\item[(integrable)] We say that $\overline{A}\in\aDiv_{\RR}(X)$ is \emph{integrable} if $\overline{A}$ can be written as a difference of two nef adelic $\RR$-Cartier divisors.
We denote by $\aInt_{\RR}(X)$ the $\RR$-vector space of all integrable adelic $\RR$-Cartier divisors on $X$.
\item[(ample)] We say that $\overline{A}\in\aDiv_{\RR}(X)$ is \emph{ample}, if $\overline{A}$ is nef and
\[
 \inf_{x\in X(\overline{K})}\left\{h_{\overline{A}}(x)\right\}>0.
\]
\item[(w-ample)] We say that $\overline{A}\in\aDiv_{\RR}(X)$ is \emph{weakly ample} or \emph{w-ample} for short if $\overline{A}$ is a positive $\RR$-linear combination $\sum_{i=1}^la_i\overline{A}_i$ of adelic Cartier divisors $\overline{A}_i$ such that each $A_i$ is ample and $H^0(mA_i)$ is generated by $\widehat{\Gamma}^{\rm ss}\left(m\overline{A}_i\right)$ for every $m\gg 1$ (see \cite{IkomaRem}).
\item[(big)] We say that $\left(\overline{D};E\right)\in\aDDiv_{\RR,\RR}(X)$ is \emph{big} if $\avol\left(\overline{D};E\right)>0$.
\item[(pseudo-effective)] We say that $\left(\overline{D};E\right)\in\aDDiv_{\RR,\RR}(X)$ is \emph{pseudo-effective} if $\avol\left(\overline{D}+\overline{B};E\right)>0$ for every big $\overline{B}\in\aDiv_{\RR}(X)$.
\end{description}

\item
There exists a unique multilinear map
\begin{align*}
 \adeg\colon  \aInt_{\RR}(X)^{\times\dim X}\times\aDiv_{\RR}(X) &\to\RR,\\
 \left(\overline{D}_1,\dots,\overline{D}_{\dim X+1}\right) &\mapsto\adeg\left(\overline{D}_1\cdots\overline{D}_{\dim X+1}\right)
\end{align*}
extending the arithmetic intersection numbers of Hermitian line bundles and having the following properties (see \cite{MoriwakiAdelic}):
\begin{enumerate}
\item The restriction $\adeg\colon \aInt_{\RR}(X)^{\times(\dim X+1)}\to\RR$ is symmetric.
\item If $\overline{D}_1,\dots,\overline{D}_{\dim X}$ are nef and $\overline{D}_{\dim X+1}$ is pseudo-effective, then
\[
 \adeg\left(\overline{D}_1\cdots\overline{D}_{\dim X+1}\right)\geq 0.
\]
\end{enumerate}

\item
Let $Y$ be a closed subvariety of $X$.
\begin{description}
\item[($Y$-effective)] We say that $\left(\overline{D};E\right)\in\aDDiv_{\RR,\RR}(X)$ is \emph{$Y$-effective} if $\left(\overline{D};E\right)\geq 0$ and $Y\not\subset\Supp(D-E)$.
We say that $\left(\overline{D};E\right)$ is \emph{strictly $Y$-effective} if $\left(\overline{D};E\right)$ is strictly effective and $Y$-effective.
We denote $\left(\overline{D};E\right)\geq_Y 0$ (respectively, $\left(\overline{D};E\right)>_Y0$) if $\left(\overline{D};E\right)$ is $Y$-effective (respectively, strictly $Y$-effective).
\item[($Y$-big)] We say that $\left(\overline{D};E\right)\in\aDDiv_{\RR,\RR}(X)$ is \emph{$Y$-big} if there exists a w-ample adelic $\RR$-Cartier divisor $\overline{A}$ such that $\left(\overline{D}-\overline{A};E\right)>_Y0$.
\item[($Y$-pseudo-effective)] We say that $\left(\overline{D};E\right)\in\aDDiv_{\RR,\RR}(X)$ is \emph{$Y$-pseudo-effective} if $\left(\overline{D}+\overline{B};E\right)$ is $Y$-big for every $Y$-big $\overline{B}\in\aDiv_{\RR}(X)$.
We denote $\left(\overline{D};E\right)\succeq_Y 0$ if $\left(\overline{D};E\right)$ is $Y$-pseudo-effective.
\end{description}
\end{enumerate}

%%%
\section{Arithmetic restricted volumes}

\subsection{Adelically normed vector spaces}

Let $K$ denote a number field.
An \emph{adelically normed $K$-vector space}
\[
\overline{V}\coloneqq\left(V,(\|\cdot\|_v^{\overline{V}})_{v\in M_K}\right)
\]
is a couple of a finite-dimensional $K$-vector space $V$ and a family of norms $(\|\cdot\|_v^{\overline{V}})_{v\in M_K}$ having the following properties:
\begin{enumerate}
\renewcommand{\labelenumi}{(\alph{enumi})}
\item For every $v\in M_K^{\rm f}$, $\|\cdot\|_v^{\overline{V}}$ is a non-Archimedean norm on $V\otimes_KK_v$.
\item $\|\cdot\|_{\infty}^{\overline{V}}$ is an Archimedean norm on $V\otimes_{\QQ}\CC$.
\item For each $a\in V$, $\|a\|_v^{\overline{V}}\leq 1$ for all but finitely many $v\in M_K^{\rm f}$.
\item Both of
\begin{equation}
\widehat{\Gamma}^{\rm s}\left(\overline{V}\right)\coloneqq\left\{a\in V\,\colon \,\text{$\|a\|_v^{\overline{V}}\leq 1$ for all $v\in M_K$}\right\}
\end{equation}
and
\begin{equation}
\widehat{\Gamma}^{\rm ss}\left(\overline{V}\right)\coloneqq\left\{a\in\widehat{\Gamma}^{\rm s}(\overline{V})\,\colon \,\|a\|_{\infty}^{\overline{V}}<1\right\}
\end{equation}
are finite sets.
\end{enumerate}
Given an adelically normed $K$-vector space $\overline{V}$ and a real number $\lambda\in\RR$, we set
\[
\|\cdot\|_v^{\overline{V}(\lambda)}\coloneqq\begin{cases} \|\cdot\|_v^{\overline{V}} & \text{if $v\in M_K^{\rm f}$,} \\ \exp(-\lambda)\|\cdot\|_v^{\overline{V}} & \text{if $v=\infty$,} \end{cases}
\]
and set $\overline{V}(\lambda)\coloneqq\left(V,(\|\cdot\|_v^{\overline{V}(\lambda)})\right)$.

\begin{remark}\label{rem:Yuan_rescale}
Yuan \cite[Lemma~2.9]{Yuan09} has proved the following estimate (see also \cite[Lemma~1.2.2]{MoriwakiEst}):
Let $\ast=\text{\rm ss}$ or $\text{\rm s}$. For any $\lambda\in\RR_{\geq 0}$, one has
\[
0\leq \log\left(\#\widehat{\Gamma}^{\ast}\left(\overline{V}(\lambda)\right)\right)-\log\left(\#\widehat{\Gamma}^{\ast}\left(\overline{V}\right)\right)\leq (\lambda+\log(3))\rk V.
\]
\end{remark}

\begin{remark}\label{rem:exact_sequence}
Let $\ast=\text{\rm ss}$ or $\text{\rm s}$.
Let $\overline{V}$ be an adelically normed $K$-vector space and let
\[
0\to V'\to V\xrightarrow{r} V''\to 0
\]
be an exact sequence of $K$-vector spaces.
We endow $V'$ with the subspace norms induced from $\overline{V}$.
\begin{enumerate}
\item\label{item: Yuan exact sequences} One has
\[
\log\left(\#\widehat{\Gamma}^{\ast}\left(\overline{V}\right)\right)\leq\log\left(\#\widehat{\Gamma}^{\ast}\left(\overline{V}'(\log(2))\right)\right)+\log\left(\#r\left(\widehat{\Gamma}^{\ast}\left(\overline{V}\right)\right)\right)
\]
and
\[
\log\left(\#\widehat{\Gamma}^{\ast}\left(\overline{V}(\log(2))\right)\right)\geq\log\left(\#\widehat{\Gamma}^{\ast}\left(\overline{V}'\right)\right)+\log\left(\#r\left(\widehat{\Gamma}^{\ast}\left(\overline{V}\right)\right)\right)
\]
(see \cite[Proposition~2.8]{Yuan09} or \cite[Lemma 1.2.2]{MoriwakiEst}).
\item\label{item: exact sequences combined} Combining Remark~\ref{rem:Yuan_rescale} and the assertion \eqref{item: Yuan exact sequences} above, one has
\begin{align*}
-\log(6)\rk V \leq\log\left(\#\widehat{\Gamma}^{\ast}\left(\overline{V}\right)\right)-\log\left(\#\widehat{\Gamma}^{\ast}\left(\overline{V}'\right)\right) &-\log\left(\#r\left(\widehat{\Gamma}^{\ast}\left(\overline{V}\right)\right)\right) \\
&\qquad\qquad \leq\log(6)\rk V'.
\end{align*}
\end{enumerate}
\end{remark}

We generalize Remark~\ref{rem:exact_sequence} \eqref{item: Yuan exact sequences} in two ways: Lemmas~\ref{lem:filtration_est} and \ref{lem:quot_exact} below, which play key roles in proving Theorem~A (see sections \ref{subsec: estimation of NO bodies} and \ref{subsec: upper bound}, respectively).

\begin{lemma}\label{lem:filtration_est}
Let $l\in\ZZ_{\geq 1}$, and let $\ast$ be either $\text{\rm ss}$ or $\text{\rm s}$.
Let $\overline{V}$ be an adelically normed $K$-vector space, and let
\[
V=V_1\supset V_2\supset\dots\supset V_{l+1}=\{0\}
\]
be a filtration of $V$.
We endow each $V_n$ with the subspace norm induced from $\overline{V}$, and denote the natural projection by $r_n\colon V_n\to V_n/V_{n+1}$ for each $n$.
We then have
\[
\log\left(\#\widehat{\Gamma}^{\ast}\left(\overline{V}(\log(l))\right)\right)\geq\sum_{n=1}^l\log\left(\# r_n\left(\widehat{\Gamma}^{\ast}\left(\overline{V}_n\right)\right)\right).
\]
\end{lemma}

\begin{proof}
Let $\ast$ denote $\text{\rm ss}$ (respectively, $\text{\rm s}$).
For each $n$, we fix a section $\sigma_n\colon r_n\left(\widehat{\Gamma}^{\ast}\left(\overline{V}_n\right)\right)\to\widehat{\Gamma}^{\ast}\left(\overline{V}_n\right)$ of the surjection $r_n\colon \widehat{\Gamma}^{\ast}\left(\overline{V}_n\right)\to r_n\left(\widehat{\Gamma}^{\ast}\left(\overline{V}_n\right)\right)$.
The required inequality follows from injectivity of the map
\begin{equation}\label{eqn:filtration_est1}
\prod_{n=1}^lr_n\left(\widehat{\Gamma}^{\ast}\left(\overline{V}_n\right)\right)\to\widehat{\Gamma}^{\ast}\left(\overline{V}(\log(l))\right),\quad \left(a_1',\dots,a_l'\right)\mapsto \sum_{n=1}^l\sigma_n\left(a_n'\right).
\end{equation}
Indeed, we have
\[
\left\|\sum_{n=1}^l\sigma_n\left(a_n'\right)\right\|_{\infty}^{\overline{V}}<l\quad \text{(respectively, $\leq l$)}
\]
for any $\left(a_1',\dots,a_l'\right)\in\prod_{n=1}^lr_n\left(\widehat{\Gamma}^{\ast}(\overline{V}_n)\right)$, which assures the existence of the map \eqref{eqn:filtration_est1}.
If we assume
\[
\sum_{n=1}^l\sigma_n(r_n(a_{1n}))=\sum_{n=1}^l\sigma_n(r_n(a_{2n}))
\]
for $(a_{11},\dots,a_{1l}),(a_{21},\dots,a_{2l})\in\prod_{n=1}^l\widehat{\Gamma}^{\ast}(\overline{V}_n)$, then we have inductively
\begin{align*}
r_1(a_{11})-r_1(a_{21}) &=r_1\left(\sum_{n=1}^l\sigma_n(r_n(a_{1n}))-\sum_{n=1}^l\sigma_n(r_n(a_{2n}))\right)=0, \\
r_2(a_{12})-r_2(a_{22})&=r_2\left(\sum_{n=2}^l\sigma_n(r_n(a_{1n}))-\sum_{n=2}^l\sigma_n(r_n(a_{2n}))\right)=0, \\
&\vdots \\
r_l(a_{1l})-r_l(a_{2l})&=r_l\left(\sigma_l(r_l(a_{1l}))-\sigma_l(r_l(a_{2l}))\right)=0.
\end{align*}
\end{proof}

\begin{lemma}\label{lem:quot_exact}
Let $\ast$ either $\text{\rm ss}$ or $\text{\rm s}$.
Let $\overline{V}$ be an adelically normed $K$-vector space, and let
\[
\xymatrix{
 0 \ar[r] & \widetilde{W} \ar[r] & \widetilde{V} \ar[r] & V' \ar[r] & 0 \\
 0 \ar[r] & W \ar[r] \ar[u]^-{r|_W} & V \ar[r]^-{r'} \ar[u]^-{r} & V' \ar[r] \ar@{=}[u] & 0
}
\]
be a commutative diagram of $K$-vector spaces where the upper and the lower sequences are respectively exact.
We endow $W$ with the subspace norm induced from $\overline{V}$.
We then have
\[
\log\left(\#r\left(\widehat{\Gamma}^{\ast}\left(\overline{V}\right)\right)\right)\leq\log\left(\#r\left(\widehat{\Gamma}^{\ast}\left(\overline{W}(\log(2))\right)\right)\right)+\log\left(\#r'\left(\widehat{\Gamma}^{\ast}\left(\overline{V}\right)\right)\right)
\]
and
\[
\log\left(\#r\left(\widehat{\Gamma}^{\ast}\left(\overline{V}(\log(2))\right)\right)\right)\geq\log\left(\#r\left(\widehat{\Gamma}^{\ast}\left(\overline{W}\right)\right)\right)+\log\left(\#r'\left(\widehat{\Gamma}^{\ast}\left(\overline{V}\right)\right)\right).
\]
\end{lemma}

\begin{proof}
Let $\ast$ denote $\text{\rm ss}$ (respectively, $\text{\rm s}$), and fix a section $\sigma\colon r'\left(\widehat{\Gamma}^{\ast}\left(\overline{V}\right)\right)\to\widehat{\Gamma}^{\ast}\left(\overline{V}\right)$ of the surjection $r'\colon \widehat{\Gamma}^{\ast}\left(\overline{V}\right)\to r'\left(\widehat{\Gamma}^{\ast}\left(\overline{V}\right)\right)$.

The first inequality follows from the fact that the image of the map
\begin{equation}
r\left(\widehat{\Gamma}^{\ast}\left(\overline{W}(\log(2))\right)\right)\times r'\left(\widehat{\Gamma}^{\ast}\left(\overline{V}\right)\right) \to\widetilde{V}, \quad \left(b',a'\right) \mapsto b'+r\left(\sigma\left(a'\right)\right),
\end{equation}
contains $r\left(\widehat{\Gamma}^{\ast}\left(\overline{V}\right)\right)$.
Indeed, given any $r(a)\in r\left(\widehat{\Gamma}^{\ast}\left(\overline{V}\right)\right)$ with $a\in\widehat{\Gamma}^{\ast}\left(\overline{V}\right)$, we have $a-\sigma(r'(a))\in W$ and $\|a-\sigma(r'(a))\|_{\infty}^{\overline{V}}< 2$ (respectively, $\leq 2$).
Hence, $a-\sigma(r'(a))\in\widehat{\Gamma}^{\ast}\left(\overline{W}(\log(2))\right)$ and
\[
r(a)=r(a-\sigma(r'(a)))+r(\sigma(r'(a))).
\]

Similarly, the second follows from the fact that the map
\begin{equation}
r\left(\widehat{\Gamma}^{\ast}\left(\overline{W}\right)\right)\times r'\left(\widehat{\Gamma}^{\ast}\left(\overline{V}\right)\right) \to r\left(\widehat{\Gamma}^{\ast}\left(\overline{V}(\log(2))\right)\right),\quad \left(b',a'\right) \mapsto b'+r\left(\sigma\left(a'\right)\right),
\end{equation}
is injective.
Indeed, given any $b\in\widehat{\Gamma}^{\ast}\left(\overline{W}\right)$ and $a'\in r'\left(\widehat{\Gamma}^{\ast}\left(\overline{V}\right)\right)$, we have $\|b+\sigma(a')\|_{\infty}^{\overline{V}}<2$ (respectively, $\leq 2$), which assures the existence of the above map.
If $r(b_1)+r(\sigma(r'(a_1)))=r(b_2)+r(\sigma(r'(a_2)))$ for $b_1,b_2\in\widehat{\Gamma}^{\ast}\left(\overline{W}\right)$ and $a_1,a_2\in\widehat{\Gamma}^{\ast}\left(\overline{V}\right)$, then
\[
r'(a_1)-r'(a_2)=r'(b_1)-r'(b_2)=0
\]
and $r(b_1)=r(b_2)$.
\end{proof}

\subsection{Estimation of Green functions}

In this subsection, $M$ denotes a equidimensional complex projective manifold.
Let $\overline{A}=\left(A,g^{\overline{A}}\right)$ be a \emph{$C^{\infty}$-metrized $\RR$-Cartier divisor} on $M$: namely a couple of an $\RR$-Cartier divisor $A$ on $M$ and an $A$-Green function $g^{\overline{A}}$ on $M$.
Assume that $A$ is ample, and that the curvature form $c_1\left(\overline{A}\right)$ is positive pointwise on $M$.
Moreover, let $\overline{D}$ be a $C^{\infty}$-metrized $\RR$-Cartier divisor on $M$, and let $\overline{\bm{E}}\coloneqq \left(\overline{E}_1,\dots,\overline{E}_l\right)$ be a family of $C^{\infty}$-metrized $\RR$-Cartier divisors on $M$ such that $E_1,\dots,E_l$ are effective.
We denote $\bm{E}\coloneqq\left(E_1,\dots,E_l\right)$, and use the dot-product notation as in Notation and terminology~\ref{NC:InnProd}.

We choose an $a_0\in\RR_{>0}$ such that
\[
\overline{A}+t\overline{D}-\sum_{i=1}^lr_i\overline{E}_i
\]
is ample with pointwise positive curvature form for every $t\in\RR$ and $\bm{r}\coloneqq (r_1,\dots,r_l)\in\RR^l$ with $|t|+\|\bm{r}\|\leq a_0$.
By \cite[Theorem~4.6]{MoriwakiZar} (see also \cite[Theorem~3.4]{Berman09} or \cite[Theorem~1.4]{Berman_Demailly}), given any $t\in\RR$ and $\bm{r}\in (\RR_{\geq 0})^l$ with $|t|+\|\bm{r}\|\leq a_0$, the set of $(A+tD-\bm{r}\cdot\bm{E})$-Green functions,
\begin{align*}
&G_{\textup{PSH}\cap C^0}\left(A+tD-\bm{r}\cdot\bm{E}\right)_{\leq g^{\overline{A}+t\overline{D}}} \\
&\quad\coloneqq\left\{g\leq g^{\overline{A}+t\overline{D}}\,\colon\,\text{$g$ is an $(A+tD-\bm{r}\cdot\bm{E})$-Green function of $(\textup{PSH}\cap C^0)$-type}\right\},
\end{align*}
admits a unique maximal element $g^{\overline{A+tD-\bm{r}\cdot\bm{E}}^{\rm (env)}}$ (see \cite[section~4]{MoriwakiZar} for the notation).
The following lemma plays an essential role in the proof of Theorem~B (see sections~\ref{subsec:Cont_at_the_boundary} and \ref{subsec:Yuan}).

\begin{lemma}\label{lem:PSH_Env}
For any $\varepsilon\in\RR$ with $0<\varepsilon\leq 1$, there exists a $\lambda_{\varepsilon}\in\RR_{>0}$ such that, for any $t\in\RR$ and $\bm{r}\in(\RR_{\geq 0})^l$ with $|t|+\|\bm{r}\|\leq a_0$, one has
\[
g^{\overline{A}+t\overline{D}-\bm{r}\cdot\overline{\bm{E}}}\leq g^{\overline{A+tD-\bm{r}\cdot\bm{E}}^{\rm (env)}}\leq g^{\overline{A}+t\overline{D}-\bm{r}\cdot\overline{\bm{E}}}+\left(\varepsilon +\lambda_{\varepsilon}\|\bm{r}\|\right).
\]
\end{lemma}

\begin{proof}
As $g^{\overline{A}+t\overline{D}-\bm{r}\cdot\overline{\bm{E}}}$ is an $(A+tD-\bm{r}\cdot\bm{E})$-Green function of $(\textup{PSH}\cap C^{\infty})$-type, one has
\[
g^{\overline{A}+t\overline{D}-\bm{r}\cdot\overline{\bm{E}}}\leq g^{\overline{A+tD-\bm{r}\cdot\bm{E}}^{\rm (env)}}
\]
for any $t\in\RR$ and $\bm{r}\in(\RR_{\geq 0})^l$ with $|t|+\|\bm{r}\|\leq a_0$.

For each $x\in M$, we can choose an open neighborhood $U_x$ of $x$ on which $g^{\overline{A}}$ (respectively, $g^{\overline{D}}$, $g^{\overline{E}_i}$) can be written as
\begin{align*}
&g^{\overline{A}}=u_{A,x}-\log(|f_A|) \\
\text{(respectively,}\qquad &g^{\overline{D}}=u_{D,x}-\log(|f_D|),\quad g^{\overline{E}_i}=u_{E_i,x}-\log(|f_{E_i}|)\qquad\text{),}
\end{align*}
where $f_A$ (respectively, $f_D$, $f_{E_i}$) denotes a local equation defining $A$ (respectively, $D$, $E_i$) on $U_x$ and $u_{A,x}$ (respectively, $u_{D,x}$, $u_{E_i,x}$) denotes a smooth plurisubharmonic function on $U_x$.
By shrinking $U_x$ if it is necessary, we may assume that
\begin{align*}
&\lvert u_{A,x}-u_{A,x}(x)\rvert\leq \frac{\varepsilon}{4} \\
\text{(respectively,}\qquad &\lvert u_{D,x}-u_{D,x}(x)\rvert\leq\frac{\varepsilon}{4a_0},\quad \lvert u_{E_i,x}-u_{E_i,x}(x)\rvert\leq 1\qquad\text{),}
\end{align*}
holds on $U_x$.

Let $g\in G_{\textup{PSH}\cap C^0}(A+tD-\bm{r}\cdot\bm{E})_{\leq g^{\overline{A}+t\overline{D}}}$, and write
\[
g=v_x-\log(|f_A|)-t\log(|f_D|)+\sum_{i=1}^lr_i\log(|f_{E_i}|)
\]
on $U_x$, where $v_x$ denotes a continuous plurisubharmonic function on $U_x$.
Then, the condition $g\leq g^{\overline{A}+t\overline{D}}$ is equivalent to
\[
v_x\leq u_{A,x}+tu_{D,x}-\sum_{i=1}^lr_i\log(|f_{E_i}|).
\]
Hence, by \cite[Lemma~4.1]{MoriwakiZar}, we can find an open neighborhood $V_x\subset U_x$ of $x$ and a constant $\lambda_x\in\RR_{>0}$, which depend only on $E_1,\dots,E_l$ and $U_x$, such that
\[
v_x\leq \left(u_{A,x}(x)+tu_{D,x}(x)+\frac{\varepsilon}{2}\right)+\lambda_x\|\bm{r}\|\leq u_{A,x}+tu_{D,x}+(\varepsilon+\lambda_x\|\bm{r}\|),
\]
thus,
\[
g\leq g^{\overline{A}+t\overline{D}-\bm{r}\cdot\overline{\bm{E}}}+\varepsilon+\sum_{i=1}^lr_i(\lambda_x+u_{E_i,x}(x)+1)
\]
holds on $V_x$ for every $g\in G_{\textup{PSH}\cap C^0}(A+tD-\bm{r}\cdot\bm{E})_{\leq g^{\overline{A}+t\overline{D}}}$.

Since $M$ is compact, we can find a finite number of points $x_1,\dots,x_n\in M$ such that $(V_{x_i})_i$ covers $M$.
Set
\[
\lambda_{\varepsilon}\coloneqq\max_{i,j}\left\{\lambda_{x_j}+u_{E_i,x_j}(x_j)+1\right\}.
\]
Then, we have
\[
g^{\overline{A+tD-\bm{r}\cdot\bm{E}}^{\rm (env)}}\leq g^{\overline{A}+t\overline{D}-\bm{r}\cdot\overline{\bm{E}}}+(\varepsilon+\lambda_{\varepsilon}\|\bm{r}\|)
\]
as desired.
\end{proof}

\subsection{Arithmetic restricted volumes}\label{subsec:ARV}

Let $X$ be a normal, projective, and geometrically connected $K$-variety, and let $\left(\overline{D};E\right)\in\aDDiv_{\RR,\RR}(X)$.
For each $v\in M_K$ and $\phi\in H^0(D-E)\otimes_KK_v$, the $D$-Green function $g_v^{\overline{D}}$ defines a metric as
\begin{equation}
|\phi|_v^{\overline{D}}(x)\coloneqq |\phi|(x)\exp\left(-g_v^{\overline{D}}(x)\right)=\exp\left(-g_v^{\overline{D}+\widehat{(\phi)}}(x)\right)
\end{equation}
for $x\in X_v^{\rm an}$, and the supremum norm $\|\cdot\|_{v,\sup}^{\overline{D}}$ on $H^0(D-E)\otimes_KK_v$ as
\begin{equation}
\|\phi\|_{v,\sup}^{\overline{D}}\coloneqq\sup_{x\in X_v^{\rm an}}\left\{|\phi|_v^{\overline{D}}(x)\right\}
\end{equation}
(see Notation and terminology~\ref{NC:Green}).
We set
\begin{align}
&\widehat{\Gamma}^{\rm f}\left(\overline{D};E\right)\coloneqq\left\{\phi\in H^0(D-E)\,\colon \,\|\phi\|_{v,\sup}^{\overline{D}}\leq 1,\,\forall v\in M_K^{\rm f}\right\}, \\
&\widehat{\Gamma}^{\rm s}\left(\overline{D};E\right)\coloneqq\left\{\phi\in\widehat{\Gamma}^{\rm f}\left(\overline{D};E\right)\,\colon \,\|\phi\|_{\infty,\sup}^{\overline{D}}\leq 1\right\},
\end{align}
and
\begin{equation}
\widehat{\Gamma}^{\rm ss}\left(\overline{D};E\right)\coloneqq\left\{\phi\in\widehat{\Gamma}^{\rm f}\left(\overline{D};E\right)\,\colon \,\|\phi\|_{\infty,\sup}^{\overline{D}}<1\right\}
\end{equation}
(see Notation and terminology~\ref{NC:Vol}).

Let $Y$ be a closed subscheme of $X$, and assume that $\left(\overline{D};E\right)\in\aDDiv_{\ZZ,\RR}(X)$ and $E\geq 0$.
We set
\begin{equation}
H^0_{X|Y}(D)\coloneqq\Image\left(H^0(D)\to H^0\left(\OO_X(D)|_Y\right)\right)
\end{equation}
and
\begin{equation}
\widehat{\Gamma}^{\ast}_{X|Y}\left(\overline{D};E\right)\coloneqq\Image\left(\widehat{\Gamma}^{\ast}\left(\overline{D};E\right)\to H^0\left(\left.\OO_X\left(D-\lceil E\rceil\right)\right|_Y\right)\right)
\end{equation}
for $\ast=\text{\rm f}$, $\text{\rm s}$, and $\text{\rm ss}$.
Here $\lceil E\rceil$ denotes the round-up Weil divisor of $E$.
In the rest of this paper, we only treat the case where $\lceil E\rceil$ is also a Cartier divisor.
We set
\begin{equation}
\ah_{X|Y}\left(\overline{D};E\right)\coloneqq\log\left(\#\widehat{\Gamma}^{\ast}_{X|Y}\left(\overline{D};E\right)\right)
\end{equation}
for $\ast=\text{\rm s}$ and $\text{\rm ss}$.

\begin{definition}\label{defn:arithmetic_restricted_vol1}
Let $\ast$ be either ${\rm ss}$ or ${\rm s}$, and let $\left(\overline{D};E\right)\in\aDDiv_{\ZZ,\ZZ}(X)$ with $E\geq 0$.
\begin{enumerate}
\item We define the \emph{CL-hull} of the finite set $\widehat{\Gamma}^{\ast}_{X|Y}\left(\overline{D};E\right)$ as the intersection
\begin{equation}
 \widehat{\Gamma}^{\ast}_{CL(X|Y)}\left(\overline{D};E\right)\coloneqq\left\langle\widehat{\Gamma}^{\ast}_{X|Y}\left(\overline{D};E\right)\right\rangle_{\ZZ}\cap\Conv_{\RR}\left(\widehat{\Gamma}^{\ast}_{X|Y}\left(\overline{D};E\right)\right),
\end{equation}
where $\left\langle\widehat{\Gamma}^{\ast}_{X|Y}\left(\overline{D};E\right)\right\rangle_{\ZZ}$ denotes the $\ZZ$-submodule generated by $\widehat{\Gamma}^{\ast}_{X|Y}\left(\overline{D};E\right)$ and $\Conv_{\RR}\left(\widehat{\Gamma}^{\ast}_{X|Y}\left(\overline{D};E\right)\right)$ denotes the convex hull of $\widehat{\Gamma}^{\ast}_{X|Y}\left(\overline{D};E\right)$ in the $\RR$-vector space generated by $\widehat{\Gamma}^{\ast}_{X|Y}\left(\overline{D};E\right)$.
We set
\begin{equation}
\ah_{CL(X|Y)}\left(\overline{D};E\right)\coloneqq\log\left(\#\widehat{\Gamma}^{\ast}_{CL(X|Y)}\left(\overline{D};E\right)\right),
\end{equation}
and define the \emph{arithmetic restricted volume of $(\overline{D};E)$ along $Y$} as
\begin{equation}\label{eqn:defavolqbase}
 \avolq{X|Y}(\overline{D};E)\coloneqq\limsup_{\substack{m\in\ZZ_{\geq 1}, \\ m\to\infty}}\frac{\ahss_{CL(X|Y)}\left(m\overline{D};mE\right)}{m^{\dim Y+1}/(\dim Y+1)!}.
\end{equation}
\item We endow $H^0_{X|Y}(D-E)\otimes_{\QQ}\RR$ with the quotient norm $\|\cdot\|_{\infty,\sup,\quot(X|Y)}^{\overline{D}}$ induced from $\left(H^0(D-E),\|\cdot\|_{\infty,\sup}^{\overline{D}}\right)$, and set
\[
\widehat{\Gamma}^{\rm ss}_{\quot(X|Y)}\left(\overline{D};E\right)\coloneqq\left\{\phi\in\widehat{\Gamma}^{\rm f}_{X|Y}\left(\overline{D};E\right)\,\colon \,\|\phi\|_{\infty,\sup,\quot(X|Y)}^{\overline{D}}<1\right\}.
\]
We set
\[
\ahss_{\quot(X|Y)}\left(\overline{D};E\right)\coloneqq\log\left(\#\widehat{\Gamma}^{\rm ss}_{\quot(X|Y)}\left(\overline{D};E\right)\right),
\]
and define
\begin{equation}
\avolq{\quot(X|Y)}\left(\overline{D};E\right)\coloneqq\limsup_{\substack{m\in\ZZ_{\geq 1}, \\ m\to\infty}}\frac{\ahss_{\quot(X|Y)}\left(m\overline{D};mE\right)}{m^{\dim Y+1}/(\dim Y+1)!}.
\end{equation}
\end{enumerate}
\end{definition}

\begin{remark}\label{rem:general_rule2}
Note that the arithmetic restricted volumes defined in Definition~\ref{defn:arithmetic_restricted_vol1} have the same basic properties as described in \cite[section~7]{IkomaRem}.
In particular, for $\left(\overline{D};E\right),\left(\overline{D}';E'\right)\in\aDDiv_{\ZZ,\ZZ}(X)$ with $\min\{E,E'\}\geq 0$,
if $s\in\widehat{\Gamma}^{\ast}_{CL(X|Y)}\left(\overline{D};E\right)$ and $s'\in\widehat{\Gamma}^{\rm s}_{X|Y}\left(\overline{D}';E'\right)$, then $s\cdot s'\in\widehat{\Gamma}^{\ast}_{CL(X|Y)}\left(\overline{D}+\overline{D}';E+E'\right)$.
\end{remark}

\begin{lemma}\label{lem: Estimates II generation}
Let $\overline{A}$ be a w-ample adelic Cartier divisor on $X$, and let $\bm{E}\coloneqq\left(E_1,\dots,E_l\right)$ be a family of effective Cartier divisors on $X$.
There then exist an $\varepsilon_0\in\RR_{>0}$, an $r_0\in\QQ$ with $0<r_0\leq 1$, and a $p_0\in\ZZ_{\geq 1}$ such that
\[
H^0(pA-\bm{q}\cdot\pmb{E})=\left\langle\widehat{\Gamma}^{\rm ss}\left(p\left(\overline{A}(-\varepsilon_0)\right);\bm{q}\cdot\bm{E}\right)\right\rangle_{\QQ}
\]
for every $(p,\bm{q})\in(\ZZ_{\geq 0})^{l+1}$ with $p\geq p_0$ and $\|\bm{q}\|\leq r_0p$.
\end{lemma}

\begin{proof}
We take any $\varepsilon_0\in\RR_{>0}$ such that $\overline{A}'\coloneqq\overline{A}(-\varepsilon_0)$ is also w-ample, and choose an $a_0\in\ZZ_{\geq 1}$ such that $a_0A-E_i$ is ample for every $i$.
We set
\[
V\coloneqq\mathcal{O}_X(A)\oplus\bigoplus_{i=1}^l\mathcal{O}_X\left(a_0A-E_i\right),
\]
and consider the projective bundle
\[
\widetilde{X}\coloneqq\PP(V)=\Proj_{\mathcal{O}_X}\left(\Sym (V)\right).
\]
By \cite[Lemma~2.3.2]{LazarsfeldI}, the tautological sheaf $\mathcal{O}_{\widetilde{X}}(1)$ is also ample.
Hence we can choose an $m_0\in\ZZ_{\geq 1}$ such that the section algebra
\[
\bigoplus_{m=0}^{\infty}H^0\left(\mathcal{O}_{\widetilde{X}}(m)\right)
\]
is generated by sections in
\[
H^0\left(\mathcal{O}_{\widetilde{X}}(m)\right)=\bigoplus_{q_0+\|\bm{q}\|=m}H^0\left(\left(q_0+a_0\|\bm{q}\|\right)A-\bm{q}\cdot\pmb{E}\right)
\]
with $m\leq m_0$.
We can choose a sufficiently large integer multiple $b_0$ of $m_0$ such that
\[
H^0(pA-\bm{q}\cdot\pmb{E})=\left\langle\widehat{\Gamma}^{\rm ss}\left(p\overline{A};\bm{q}\cdot\pmb{E}\right)\right\rangle_{\QQ}
\]
for every $p\geq b_0$ and every $\bm{q}\in(\ZZ_{\geq 0})^l$ with $\|\bm{q}\|\leq m_0$.
Thus, if we take $p_0\geq b_0$ and $r_0\leq 1/(a_0m_0+2b_0)$, then we obtain the desired assertion.
\end{proof}

\begin{proposition}\label{prop: Estimates II inclusions}
Let $Y$ be a prime Cartier divisor on $X$, let $\overline{A}$ be a w-ample adelic Cartier divisor on $X$, and let $\bm{E}\coloneqq\left(E_1,\dots,E_l\right)$ be a family of effective Cartier divisors on $X$ such that $\ord_Y(E_j)=0$ for every $j$.
There then exists a $\rho_0\in\QQ_{>0}$ such that, given any $\varepsilon\in\RR_{>0}$, one can find a $p_0(\varepsilon)\in\ZZ_{\geq 1}$, which depends on $\overline{A}$, $\overline{\bm{D}}$, $\bm{E}$, and $\varepsilon$, such that
\[
\widehat{\Gamma}^{\rm ss}_{\quot(X|Y)}\left(p\overline{A};rY+\bm{q}\cdot\bm{E}\right)\subset\widehat{\Gamma}^{\rm ss}_{X|Y}\left(p\left(\overline{A}(\varepsilon)\right);rY+\bm{q}\cdot\bm{E}\right)
\]
and
\[
\widehat{\Gamma}^{\rm ss}_{\quot(X|Y)}\left(p\left(\overline{A}(-\varepsilon)\right);rY+\bm{q}\cdot\bm{E}\right)\subset\widehat{\Gamma}^{\rm ss}_{X|Y}\left(p\overline{A};rY+\bm{q}\cdot\bm{E}\right)
\]
for every $(p,\bm{q},r)\in(\ZZ_{\geq 0})^{l+2}$ with $p\geq p_0(\varepsilon)$ and $r+\|\bm{q}\|\leq\rho_0p$.
\end{proposition}

\begin{proof}
By Lemma~\ref{lem: Estimates II generation}, there exist an $\varepsilon_0\in\RR_{>0}$, a $\rho_0'\in\QQ_{>0}$, and a $p_0'\in\ZZ_{\geq 1}$ such that
\[
H^0\left(pA-rY-\bm{q}\cdot\bm{E}\right)=\left\langle\widehat{\Gamma}^{\rm ss}\left(p\left(\overline{A}(-\varepsilon_0)\right);rY+\bm{q}\cdot\bm{E}\right)\right\rangle_{\QQ}
\]
for every $(p,\bm{q},r)\in(\ZZ_{\geq 0})^{l+2}$ with $p\geq p_0'$ and $r+\|\bm{q}\|\leq\rho_0'p$.
We set $\rho_0\coloneqq \rho_0'/2$, and choose a sufficiently large $p_0(\varepsilon)\geq\max\{p_0',1/\rho_0\}$ such that
\begin{align}
&I(A)p^{\dim X}e^{-p\varepsilon_0/2}\leq 1, \label{eqn: inclusion Zhang}\\
&1+I(A)p^{\dim X}e^{-p\varepsilon_0/2}\leq e^{p\varepsilon}, \label{eqn: inclusion hypothesis 1}
\end{align}
and
\begin{equation}
e^{-p\varepsilon}+I(A)p^{\dim X}e^{-p\varepsilon_0/2}\leq 1 \label{eqn: inclusion hypothesis 2}
\end{equation}
for every $p\geq p_0(\varepsilon)$.

As
\[
\widehat{\Gamma}^{\rm f}\left(p\overline{A};(r+1)Y+\bm{q}\cdot\bm{E}\right)=\widehat{\Gamma}^{\rm f}\left(p\overline{A};rY+\bm{q}\cdot\bm{E}\right)\cap H^0(pA-(r+1)Y-\bm{q}\cdot\bm{E})
\]
under the natural inclusion $H^0(pA-(r+1)Y-\bm{q}\cdot\bm{E})\subset H^0(pA-rY-\bm{q}\cdot\bm{E})$, we obtain the exact sequence
\[
0\to \widehat{\Gamma}^{\rm f}\left(p\overline{A};(r+1)Y+\bm{q}\cdot\bm{E}\right)\to \widehat{\Gamma}^{\rm f}\left(p\overline{A};rY+\bm{q}\cdot\bm{E}\right)\xrightarrow{\pi} \widehat{\Gamma}_{X|Y}^{\rm f}\left(p\overline{A};rY+\bm{q}\cdot\bm{E}\right)\to 0
\]
of $\ZZ$-modules.
Let $e_1^{p,\bm{q},r}$, $\dots$, $e_M^{p,\bm{q},r}\in\widehat{\Gamma}^{\rm ss}\left(p\left(\overline{A}(-\varepsilon_0/2)\right);(r+1)Y+\bm{q}\cdot\bm{E}\right)$ be a $\ZZ$-basis for $\widehat{\Gamma}^{\rm f}\left(p\overline{A};(r+1)Y+\bm{q}\cdot\bm{E}\right)$ as in \cite[Lemma~1.7]{ZhangSurf} (see also \eqref{eqn: inclusion Zhang}), and take $\ZZ$-linearly independent sections $f_1^{p,\bm{q},r}$, $\dots$, $f_N^{p,\bm{q},r}\in\widehat{\Gamma}^{\rm f}\left(p\overline{A};rY+\bm{q}\cdot\bm{E}\right)$ such that $\pi(f_1^{p,\bm{q},r})$, $\dots$, $\pi(f_N^{p,\bm{q},r})$ are all nonzero and generate $\widehat{\Gamma}_{X|Y}^{\rm f}\left(p\overline{A};rY+\bm{q}\cdot\bm{E}\right)$.

Each section in $\widehat{\Gamma}^{\rm ss}_{\quot(X|Y)}\left(p\overline{A};rY+\bm{q}\cdot\bm{E}\right)$ is an image via $\pi$ of a section $s\in H^0(pA-rY-\bm{q}\cdot\bm{E})\otimes_{\QQ}\RR$ that can be written in the form
\[
s=\sum_{i=1}^M\alpha_ie_i^{p,\bm{q},r}+\sum_{j=1}^N\beta_jf_j^{p,\bm{q},r}
\]
($\alpha_i\in\RR$, $\beta_j\in\ZZ$) and that satisfies $\|s\|_{\infty,\sup}^{p\overline{A}}<1$.
As
\[
\left\|\sum_{i=1}^M\lfloor\alpha_i\rfloor e_i^{p,\bm{q},r}+\sum_{j=1}^N\beta_jf_j^{p,\bm{q},r}\right\|_{\infty,\sup}^{p\overline{A}}<1+I(A)p^{\dim X}e^{-p\varepsilon_0/2}\leq e^{p\varepsilon}
\]
by \eqref{eqn: inclusion hypothesis 1}, we have $\widehat{\Gamma}^{\rm ss}_{\quot(X|Y)}\left(p\overline{A};rY+\bm{q}\cdot\bm{E}\right)\subset\widehat{\Gamma}^{\rm ss}_{X|Y}\left(p\left(\overline{A}(\varepsilon)\right);rY+\bm{q}\cdot\bm{E}\right)$.

Similarly, each section in $\widehat{\Gamma}^{\rm ss}_{\quot(X|Y)}\left(p\left(\overline{A}(-\varepsilon)\right);rY+\bm{q}\cdot\bm{E}\right)$ is an image via $\pi$ of a section $s'\in H^0(pA-rY-\bm{q}\cdot\bm{E})\otimes_{\QQ}\RR$ that can be written in the form
\[
s'=\sum_{i=1}^M\alpha_ie_i^{p,\bm{q},r}+\sum_{j=1}^N\beta_jf_j^{p,\bm{q},r}
\]
($\alpha_i\in\RR$, $\beta_j\in\ZZ$) and that satisfies $\|s'\|_{\infty,\sup}^{p\overline{A}}<e^{-p\varepsilon}$.
As
\[
\left\|\sum_{i=1}^M\lfloor\alpha_i\rfloor e_i^{p,\bm{q},r}+\sum_{j=1}^N\beta_jf_j^{p,\bm{q},r}\right\|_{\infty,\sup}^{p\overline{A}}<e^{-p\varepsilon}+I(A)p^{\dim X}e^{-p\varepsilon_0/2}\leq 1
\]
by \eqref{eqn: inclusion hypothesis 2}, we have $\widehat{\Gamma}^{\rm ss}_{\quot(X|Y)}\left(p\left(\overline{A}(-\varepsilon)\right);rY+\bm{q}\cdot\bm{E}\right)\subset\widehat{\Gamma}^{\rm ss}_{X|Y}\left(p\overline{A};rY+\bm{q}\cdot\bm{E}\right)$ as desired.
\end{proof}

\begin{lemma}\label{lem:construction_of_the_metric}
Let $\mathscr{X}$ be a normal, projective, and arithmetic variety, let $\mathscr{Y}$ be a reduced closed subscheme of $\mathscr{X}$, and let $\overline{\mathscr{D}}\coloneqq\left(\mathscr{D},g^{\overline{\mathscr{D}}}\right)$ be an arithmetic Cartier divisor on $\mathscr{X}$, and let $\bm{\mathscr{E}}=\left(\mathscr{E}_1,\dots,\mathscr{E}_l\right)$ be a family of effective Cartier divisors on $\mathscr{X}$.
There then exist another $\mathscr{D}$-Green function $g^{\overline{\mathscr{D}}'}$ and an $\mathscr{E}_i$-Green function $g^{\overline{\mathscr{E}}_i'}$ on $\mathscr{X}_{\infty}^{\rm an}$ for each $i$ such that
\[
\Image\left(\widehat{\Gamma}^{\rm s}\left(p\overline{\mathscr{D}};\bm{q}\cdot\pmb{\mathscr{E}}\right)\to H^0\left(\left.\OO_{\mathscr{X}}\left(p\mathscr{D}-\bm{q}\cdot\pmb{\mathscr{E}}\right)\right|_{\mathscr{Y}}\right)\right)\subset \widehat{\Gamma}^{\rm s}\left(\left.\OO_{\mathscr{X}}\left(p\overline{\mathscr{D}}'-\bm{q}\cdot\overline{\pmb{\mathscr{E}}}'\right)\right|_{\mathscr{Y}}\right)
\]
for every $(p,\bm{q})\in(\ZZ_{\geq 0})^{l+1}$, where $\overline{\mathscr{D}}'\coloneqq \left(\mathscr{D},g^{\overline{\mathscr{D}}'}\right)$ and $
\overline{\pmb{\mathscr{E}}}'\coloneqq\left(\left(\mathscr{E}_1,g^{\overline{\mathscr{E}}_1'}\right),\dots,\left(\mathscr{E}_l,g^{\overline{\mathscr{E}}_l'}\right)\right)$.
\end{lemma}

\begin{proof}
Given any continuous and nonnegative function $f$ on $\mathscr{X}_{\infty}^{\rm an}$ that is invariant under the complex conjugation, we have
\[
\widehat{\Gamma}^{\rm s}\left(p\overline{\mathscr{D}};\bm{q}\cdot\pmb{\mathscr{E}}\right)\subset\widehat{\Gamma}^{\rm s}\left(p\left(\overline{\mathscr{D}}(f)\right);\bm{q}\cdot\pmb{\mathscr{E}}\right)
\]
for any $(p,\bm{q})\in(\ZZ_{\geq 0})^{l+1}$ (see Notation and terminology~\ref{NC:adelic}).
Hence, we can assume that $\overline{\mathscr{D}}$ is of $C^{\infty}$-type.
Let $1_{\mathscr{E}_i}$ denote the canonical section of $\OO_{\mathscr{X}}(\mathscr{E}_i)$.
We endow each $\mathscr{E}_i$ with any $\mathscr{E}_i$-Green function $g^{\overline{\mathscr{E}_i}}$ of $C^{\infty}$-type, and consider the family $\overline{\pmb{\mathscr{E}}}\coloneqq \left(\overline{\mathscr{E}_1},\dots,\overline{\mathscr{E}_l}\right)$ of arithmetic Cartier divisors of $C^{\infty}$-type.
We choose a suitable real number $\alpha$ such that $0<\alpha<1$ and such that the open subset
\[
U\coloneqq\left\{x\in \mathscr{X}_{\infty}^{\rm an}\,\colon \,\text{$\left|1_{\mathscr{E}_i}\right|_{\infty}^{\overline{\mathscr{E}_i}}(x)>\alpha$ for every $i$ }\right\}
\]
is nonempty.
By \cite[Lemma~1.1.4]{MoriwakiCont}, there exists a constant $C>1$, depending only on $\overline{\mathscr{D}}$, $\overline{\pmb{\mathscr{E}}}$, $U$, and $\mathscr{X}$, such that
\[
C^{-(p+\|\bm{q}\|)}\|\phi\|_{\infty,\sup}^{p\overline{\mathscr{D}}-\bm{q}\cdot\overline{\pmb{\mathscr{E}}}}\leq\sup_{x\in U}\left\{|\phi|_{\infty}^{p\overline{\mathscr{D}}-\bm{q}\cdot\overline{\pmb{\mathscr{E}}}}(x)\right\}
\]
holds for any $\phi\in H^0\left(p\mathscr{D}-\bm{q}\cdot\pmb{\mathscr{E}}\right)\otimes_{\ZZ}\RR$.
Hence
\begin{align}
\left\|\phi\cdot 1_{\mathscr{E}_1}^{q_1}\cdots 1_{\mathscr{E}_l}^{q_l}\right\|_{\infty,\sup}^{p\overline{\mathscr{D}}} &\geq\sup_{x\in U}\left\{\left|\phi\cdot 1_{\mathscr{E}_1}^{q_1}\cdots 1_{\mathscr{E}_l}^{q_l}\right|_{\infty}^{p\overline{\mathscr{D}}}(x)\right\} \nonumber\\
&\geq\alpha^{\|\bm{q}\|}\cdot\sup_{x\in U}\left\{|\phi|_{\infty}^{p\overline{\mathscr{E}}-\bm{q}\cdot\overline{\pmb{\mathscr{E}}}}(x)\right\} \nonumber\\
&\geq \left(\alpha^{-1}C\right)^{-(p+\|\bm{q}\|)}\cdot\|\phi\|_{\infty,\sup}^{p\overline{\mathscr{D}}-\bm{q}\cdot\overline{\pmb{\mathscr{E}}}} \label{eqn: Construction of the metric}
\end{align}
for any $\phi\in H^0\left(p\mathscr{D}-\bm{q}\cdot\pmb{\mathscr{E}}\right)\otimes_{\ZZ}\RR$.
If we set $g^{\overline{\mathscr{D}}'}\coloneqq g^{\overline{\mathscr{D}}}+\log\left(\alpha^{-1}C\right)$ and $g^{\overline{\mathscr{E}_i}'}\coloneqq g^{\overline{\mathscr{E}_i}}-\log\left(\alpha^{-1}C\right)$ for each $i$, then
\[
\widehat{\Gamma}^{\rm s}\left(p\overline{\mathscr{D}};\bm{q}\cdot\pmb{\mathscr{E}}\right)\subset\widehat{\Gamma}^{\rm s}\left(p\overline{\mathscr{D}}'-\bm{q}\cdot\overline{\pmb{\mathscr{E}}}'\right).
\]
\end{proof}

\subsection{Yuan's estimates}\label{subsec:YuansEst}

Given $X$, $Y$ as in the previous subsection, and given a pair $\left(\overline{D},E\right)\in\aDDiv_{\ZZ,\ZZ}(X)$, we choose a model $\mathscr{X}$ (respectively, $\mathscr{Y}$) of $X$ (respectively, $Y$) and an arithmetic Cartier divisor $\overline{\mathscr{M}}$ on $\mathscr{X}$ in the following way:

By \cite[Theorem~4.1.3]{MoriwakiAdelic}, we can find a normal and projective $O_K$-model $\mathscr{X}$ of $X$ and a pair $\left(\overline{\mathscr{D}};\mathscr{E}\right)$ consisting of an arithmetic Cartier divisor $\overline{\mathscr{D}}$ on $\mathscr{X}$ and a horizontal Cartier divisor $\mathscr{E}$ on $\mathscr{X}$ such that
\[
\left(\overline{D};E\right)\leq_Y\left(\overline{\mathscr{D}};\mathscr{E}\right)^{\rm ad}
\]
(see Notation and terminology~\ref{NC:adelic}).
Let $\mathscr{Y}$ be the Zariski closure of $Y$ in $\mathscr{X}$.
By Lemma~\ref{lem:construction_of_the_metric}, there exists an arithmetic Cartier divisor $\overline{\mathscr{M}}$ on $\mathscr{X}$ such that
\begin{equation}
\widehat{\Gamma}^{\rm f}_{X|Y}\left(m\overline{D};mE\right)\subset H^0\left(\left.\mathcal{O}_{\mathscr{X}}\left(m\mathscr{M}\right)\right|_{\mathscr{Y}}\right)
\end{equation}
and
\begin{equation}
\widehat{\Gamma}^{\rm ss}_{X|Y}\left(m\overline{D};mE\right)\subset\widehat{\Gamma}^{\rm s}\left(\left.\mathcal{O}_{\mathscr{X}}\left(m\overline{\mathscr{M}}\right)\right|_{\mathscr{Y}}\right)
\end{equation}
for any $m\in\ZZ_{\geq 0}$.
In the rest of this subsection, we fix any triplet $\left(\mathscr{X},\mathscr{Y},\overline{\mathscr{M}}\right)$ satisfying the above conditions.

\begin{remark}
In the terminology of \cite{MoriwakiEst}, $\widehat{\Gamma}^{\rm ss}_{X|Y}\left(m\overline{D};mE\right)$ is an \emph{arithmetic linear series of $\left.\mathcal{O}_{\mathscr{X}}\left(m\overline{\mathscr{M}}\right)\right|_{\mathscr{Y}}$}.
\end{remark}

\begin{definition}\label{defn: good flag}
Let $\mathscr{Z}$ be a projective arithmetic variety. A \emph{good flag on $\mathscr{Z}$ over a prime number $p$} is a flag
\[
\mathscr{F}_{\geq 1}\colon \mathscr{Z}=\mathscr{F}_1\supset\mathscr{F}_2\supset\cdots\supset\mathscr{F}_{\dim\mathscr{Y}+1}
\]
on $\mathscr{Z}$ (see Notation and terminology~\ref{NC:flag}) such that the condition ($\ast$) below is satisfied.
\begin{enumerate}
\item[($\ast$)] Let $\pi\colon \mathscr{Z}\to\Spec(R)$ denote a Stein factorization of the structure morphism of $\mathscr{Z}$.
There exists a prime ideal $\mathfrak{p}$ of $R$ having the following properties:
\begin{enumerate}
\item[(a)] $R_{\mathfrak{p}}$ is a discrete valuation ring.
\item[(b)] $\mathfrak{p}\cap\ZZ=p\ZZ$ and $\FF_p\to R/\mathfrak{p}R$ is isomorphic.
\item[(c)] $\mathscr{F}_2=\pi^{-1}(\mathfrak{p})$.
\item[(d)] The closed point $\mathscr{F}_{\dim\mathscr{Y}+1}$ is regular and $\FF_p$-rational.
\end{enumerate}
\end{enumerate}
\end{definition}

Moreover, we introduce the following constants, which will be used throughout this paper.

\begin{definition}\label{defn:BasicConstants}
\begin{enumerate}
\item Given an $\RR$-Cartier divisor $N$ on $X$, we set
\begin{equation}\label{eqn:def_of_ID}
I(N)\coloneqq \sup_{m\in\ZZ_{\geq 1}}\left\{\frac{\rk_KH^0(mN)}{m^{\dim X}}\right\}.
\end{equation}
\item Given any adelic $\RR$-Cartier divisor $\overline{N}$ on $X$, we set
\begin{equation}
\delta\left(\overline{N}\right)\coloneqq \inf_{\overline{A}}\left\{\frac{\adeg\left(\overline{N}\cdot\overline{A}^{\cdot\dim X}\right)}{\vol(A)}\right\},
\end{equation}
where the infimum is taken over all nef adelic $\RR$-Cartier divisors $\overline{A}$ on $X$ such that $\vol(A)$ is positive.
\item Given an adelic $\RR$-Cartier divisor $\overline{N}$ on $X$, a prime number $p$, and an $m\in\ZZ_{\geq 0}$, we set
\begin{align}
&C\left(\overline{N},X,p,m\right)\coloneqq [K:\QQ]I(N) \nonumber\\
&\qquad\qquad \times\left(\log(4)\delta\left(\overline{N}\right)+\frac{\log(4p)\log\left(4pI(N)m^{\dim X}\right)}{m}\right).
\end{align}
\item Given an adelic $\RR$-Cartier divisor $\overline{N}$ on $X$, we set
\begin{equation}\label{eqn:def_of_Cprime}
C'\left(\overline{N},X\right)\coloneqq [K:\QQ]I(N)\delta\left(\overline{N}\right)\log(4).
\end{equation}
\end{enumerate}
\end{definition}

A result of Yuan and Moriwaki \cite{Yuan09,MoriwakiEst,IkomaRem}\footnote{In \cite[Theorem 6.7]{IkomaRem}, ``any symmetric CL-subset of $\widehat{\Gamma}^{\rm f}(\overline{L})$'' should be read as ``any symmetric CL-subset of $\widehat{\Gamma}^{\rm s}(\overline{L})$'' and, in \cite[Definition~7.2, Propositions~7.9, and 7.10]{IkomaRem}, ``$\widehat{\Gamma}^{\rm ss}_{X|Y}(m\overline{L})$'' should be read as ``$\widehat{CL}_{X|Y}(m\overline{L})$''. The author apologizes for any inconvenience.} then asserts the following:

\begin{theorem}\label{thm:YuanMoriwaki}
Let $X$ be a normal, projective, and geometrically connected $K$-variety, let $Y$ be a closed subvariety of $X$, and let $\left(\overline{D};E\right)\in\aDDiv_{\ZZ,\ZZ}(X)$ with $E\geq 0$.
Let $?$ denotes either $CL(X|Y)$ or $\quot(X|Y)$.
Choose a model $\left(\mathscr{X},\mathscr{Y},\overline{\mathscr{M}}\right)$ as above, and let $\mathscr{F}_{\geq 1}$ be a good flag on $\mathscr{Y}$ over a prime number $p$.
If $\widehat{\Gamma}^{\rm ss}_{?}\left(m\overline{D};mE\right)\neq\{0\}$, then
\begin{align*}
&\left|\ahss_{?}\left(m\overline{D};mE\right)-\#\bm{w}_{\mathscr{F}_{\geq 1}}\left(\widehat{\Gamma}^{\rm ss}_{?}\left(m\overline{D};mE\right)\setminus\{0\}\right)\log(p)\right| \\
&\qquad\qquad\qquad\qquad\qquad\qquad\qquad\quad \leq \frac{C\left(\left.\overline{\mathscr{M}}^{\rm ad}\right|_Y,Y,p,m\right)}{\log(p)}m^{\dim Y+1}
\end{align*}
for any $m\in\ZZ_{\geq 0}$.
\end{theorem}

\begin{proof}
Let $\nu\colon \mathscr{Y}'\to\mathscr{Y}$ be the relative normalization of $\mathscr{Y}$ in $Y$ and let $\nu_*^{-1}\mathscr{F}_{\geq 1}$ be the pullback of $\mathscr{F}_{\geq 1}$ via $\nu$ (see \cite[Lemma~6.3]{IkomaRem}), where we note that $\nu$ is isomorphic around the closed point $\mathscr{F}_{\dim\mathscr{Y}+1}$.
Then we obtain the result by applying \cite[Theorem~2.2]{MoriwakiEst} to $\mathscr{Y}'$, $\nu_*^{-1}\mathscr{F}_{\geq 1}$, $\left.\mathcal{O}_{\mathscr{X}}\left(\overline{\mathscr{M}}\right)\right|_{\mathscr{Y}'}$, and $\widehat{\Gamma}^{\rm ss}_?\left(m\overline{D};mE\right)$ for each $m\in\ZZ_{\geq 0}$ (see also \cite[Theorem~6.6]{IkomaRem}).
\end{proof}

\begin{corollary}[{\cite[Corollary~2.3]{MoriwakiEst}}]\label{cor:YuanMoriwaki}
We use the same notation as in Theorem~\ref{thm:YuanMoriwaki}.
We have
\begin{align*}
&\limsup_{m\to\infty}\left|\frac{\ahss_{?}\left(m\overline{D};mE\right)}{m^{\dim Y+1}}-\frac{\#\bm{w}_{\mathscr{F}_{\geq 1}}\left(\widehat{\Gamma}^{\rm ss}_{?}\left(m\overline{D};mE\right)\setminus\{0\}\right)\log(p)}{m^{\dim Y+1}}\right| \\
&\qquad\qquad\qquad\qquad\qquad\qquad\qquad\qquad\qquad\qquad\quad \leq\frac{C'\left(\left.\overline{\mathscr{M}}^{\rm ad}\right|_Y,Y\right)}{\log(p)}.
\end{align*}
\end{corollary}

\begin{definition}\label{defn: arithmetic NO bodies 1}
Let $?$ denote either $CL(X|Y)$ or $\quot(X|Y)$, and let $\left(\overline{D};E\right)\in\aDDiv_{\ZZ,\ZZ}(X)$ with $E\geq 0$.
Then we set
\[
\widehat{\Delta}_{?}^{\mathscr{F}_{\geq 1}}\left(\overline{D};E\right)\coloneqq \overline{\left(\bigcup_{m\in\ZZ_{\geq 1}}\frac{1}{m}\bm{w}_{\mathscr{F}_{\geq 1}}\left(\widehat{\Gamma}^{\rm ss}_{?}\left(m\overline{D};mE\right)\setminus\{0\}\right)\right)}.
\]
More generally, given any $\left(\overline{D};E\right)\in\aDDiv_{\QQ,\QQ}(X)$, we set
\[
\widehat{\Delta}_{?}^{\mathscr{F}_{\geq 1}}\left(\overline{D};E\right)\coloneqq \frac{1}{n}\widehat{\Delta}_{?}^{\mathscr{F}_{\geq 1}}\left(n\overline{D};nE\right),
\]
where $n$ denotes any positive integer such that $\left(n\overline{D};nE\right)\in\aDDiv_{\ZZ,\ZZ}(X)$.

If $\left(\overline{D};E\right)$ is $Y$-big, then the same arguments as in \cite[Propositions~5.1 and 5.2]{MoriwakiEst} (see also \cite[Proposition~7.7]{IkomaRem}) will lead to
\begin{equation}\label{eqn: defn: arithmetic NO bodies 1}
\vol_{\RR^{\dim Y+1}}\left(\widehat{\Delta}_{?}^{\mathscr{F}_{\geq 1}}\left(\overline{D};E\right)\right)=\lim_{\substack{m\in\ZZ_{\geq 1}, \\ m\to\infty}}\frac{\#\bm{w}_{\mathscr{F}_{\geq 1}}\left(\widehat{\Gamma}^{\rm ss}_{?}\left(m\overline{D};mE\right)\setminus\{0\}\right)}{m^{\dim Y+1}} \in\RR_{>0}.
\end{equation}
\end{definition}

\begin{corollary}\label{cor:Yuan_main}
Let $X$ be a normal, projective, and geometrically connected $K$-variety, let $Y$ be a closed subvariety of $X$, and let $\left(\overline{D};E\right),\left(\overline{D}';E'\right)\in\aDDiv_{\QQ,\QQ}(X)$ be $Y$-big pairs on $X$ with $\min\{E,E'\}\geq 0$.
Let $?$ denote either $X|Y$ or $\quot(X|Y)$.
\begin{enumerate}
\item If $\left(\overline{D};E\right)\in\aDDiv_{\ZZ,\ZZ}(X)$, then the sequence
\[
\left(\frac{\ahss_{CL(X|Y)}\left(m\overline{D};mE\right)}{m^{\dim Y+1}/(\dim Y+1)!}\right)_{m\in\ZZ_{\geq 1}}\qquad \text{(respectively, }\quad \left(\frac{\ahss_{\quot(X|Y)}\left(m\overline{D};mE\right)}{m^{\dim Y+1}/(\dim Y+1)!}\right)_{m\in\ZZ_{\geq 1}}\text{ )}
\]
converges to $\avol_{X|Y}\left(\overline{D};E\right)$ (respectively, $\avol_{\quot(X|Y)}\left(\overline{D};E\right)$).
\item\label{item: cor Yuan main homogeneity} If $\left(\overline{D};E\right)\in\aDDiv_{\ZZ,\ZZ}(X)$, then
\[
 \avol_{?}\left(a\overline{D};aE\right)=a^{\dim Y+1}\cdot\avol_{?}\left(\overline{D};E\right)
\]
for every $a\in\ZZ_{\geq 1}$.
In particular, we can define $\avolq{?}\left(\overline{D};E\right)$ for any $Y$-big pair $\left(\overline{D};E\right)\in\aDDiv_{\QQ,\QQ}(X)$.
\item\label{item: cor Yuan main BM} The Brunn--Minkowski inequality holds true for arithmetic restricted volumes:
\begin{align*}
 &\avol_{?}\left(\overline{D}+\overline{D}';E+E'\right)^{1/(\dim Y+1)} \\
 &\qquad\qquad\quad \geq\avol_{?}\left(\overline{D};E\right)^{1/(\dim Y+1)}+\avol_{?}\left(\overline{D}';E'\right)^{1/(\dim Y+1)}.
\end{align*}
\item Assume that $Y$ has codimension one in $X$.
Let $\overline{\bm{A}}\coloneqq\left(\overline{A}_1,\dots,\overline{A}_m\right)\in\aCDiv_{\QQ}(X)^{\times m}$, and let $\bm{B}\coloneqq\left(B_1,\dots,B_n\right)\in\Div_{\QQ}(X)^{\times n}$.
Assume that one of the following two conditions is satisfied.
\begin{enumerate}
\item[(a)] $\ord_Y(E)>0$.
\item[(b)] $\ord_Y(B_j)=0$ for every $j$.
\end{enumerate}
Then one has
\[
 \lim_{\substack{\pmb{\varepsilon}\in\QQ^m,\, \pmb{\delta}\in\QQ^n, \\ \pmb{\varepsilon},\,\pmb{\delta}\to 0}}\avolq{?}\left(\overline{D}+\pmb{\varepsilon}\cdot\overline{\bm{A}};E+\pmb{\delta}\cdot\bm{B}\right)=\avol_{?}\left(\overline{D};E\right).
\]
\end{enumerate}
\end{corollary}

\begin{proof}
Let $\mathscr{X}$, $\mathscr{Y}$, $\mathscr{F}_{\geq 1}$, $p$, and $\overline{\mathscr{M}}$ be as in Theorem~\ref{thm:YuanMoriwaki}.

(1): By \cite[Proposition~1.4.1]{MoriwakiEst}, there exist good flags on $\mathscr{Y}$ over infinitely many prime numbers.
Thus, given any $\varepsilon\in\RR_{>0}$, we can find a prime number $p$ such that there exists a good flag on $\mathscr{Y}$ over $p$ and
\[
 \frac{C'\left(\left.\overline{\mathscr{M}}^{\rm ad}\right|_Y,Y\right)}{\log(p)}\leq\varepsilon.
\]
Thus, by Corollary~\ref{cor:YuanMoriwaki}, we obtain
\begin{align*}
&0\leq\limsup_{\substack{m\in\ZZ_{\geq 1}, \\ m\to\infty}}\frac{\ahss_{CL(X|Y)}\left(m\overline{D};mE\right)}{m^{\dim Y+1}}-\liminf_{\substack{m\in\ZZ_{\geq 1}, \\ m\to\infty}}\frac{\ahss_{CL(X|Y)}\left(m\overline{D};mE\right)}{m^{\dim Y+1}}\leq 2\varepsilon \\
\text{(respectively,}\quad &0\leq\limsup_{\substack{m\in\ZZ_{\geq 1}, \\ m\to\infty}}\frac{\ahss_{\quot(X|Y)}\left(m\overline{D};mE\right)}{m^{\dim Y+1}}-\liminf_{\substack{m\in\ZZ_{\geq 1}, \\ m\to\infty}}\frac{\ahss_{\quot(X|Y)}\left(m\overline{D};mE\right)}{m^{\dim Y+1}}\leq 2\varepsilon\quad\text{)}
\end{align*}
(see Definition~\ref{defn: arithmetic NO bodies 1}) and conclude the proof.

The assertion (2) is a consequence of the assertion (1).

(3): By Lemma~\ref{lem:construction_of_the_metric}, one can choose an arithmetic Cartier divisor $\overline{\mathscr{N}}$ on $\mathscr{X}$ such that
\begin{align*}
&\widehat{\Gamma}_{X|Y}^{\rm f}\left(m\overline{D};mE\right)\subset H^0\left(\left.\mathcal{O}_{\mathscr{X}}\left(m\mathscr{N}\right)\right|_{\mathscr{Y}}\right), \\
&\widehat{\Gamma}_{X|Y}^{\rm f}\left(m\overline{D}';mE'\right)\subset H^0\left(\left.\mathcal{O}_{\mathscr{X}}\left(m\mathscr{N}\right)\right|_{\mathscr{Y}}\right), \\
&\widehat{\Gamma}_{X|Y}^{\rm ss}\left(m\overline{D};mE\right)\subset\widehat{\Gamma}^{\rm s}\left(\left.\mathcal{O}_{\mathscr{X}}\left(m\overline{\mathscr{N}}\right)\right|_{\mathscr{Y}}\right),
\end{align*}
and
\[
\widehat{\Gamma}_{X|Y}^{\rm ss}\left(m\overline{D}';mE'\right)\subset\widehat{\Gamma}^{\rm s}\left(\left.\mathcal{O}_{\mathscr{X}}\left(m\overline{\mathscr{N}}\right)\right|_{\mathscr{Y}}\right)
\]
for all $m$.
Let $\varepsilon\in\RR_{>0}$.
By \cite[Proposition~1.4.1]{MoriwakiEst}, there exists a prime number $p$ such that there exists a good flag on $\mathscr{Y}$ over $p$ and such that
\begin{equation}
 \frac{C'\left(\left.\overline{\mathscr{N}}^{\rm ad}\right|_Y,Y\right)}{\log(p)}\leq\varepsilon.
\end{equation}
Applying the classical Brunn--Minkowski inequality to
\begin{align*}
&\widehat{\Delta}_{CL(X|Y)}^{\mathscr{F}_{\geq 1}}\left(\overline{D};E\right)+\widehat{\Delta}_{CL(X|Y)}^{\mathscr{F}_{\geq 1}}\left(\overline{D}';E'\right)\subset\widehat{\Delta}_{CL(X|Y)}^{\mathscr{F}_{\geq 1}}\left(\overline{D}+\overline{D}';E+E'\right) \\
\text{(respectively,}\quad &\widehat{\Delta}_{\quot(X|Y)}^{\mathscr{F}_{\geq 1}}\left(\overline{D};E\right)+\widehat{\Delta}_{\quot(X|Y)}^{\mathscr{F}_{\geq 1}}\left(\overline{D}';E'\right)\subset\widehat{\Delta}_{\quot(X|Y)}^{\mathscr{F}_{\geq 1}}\left(\overline{D}+\overline{D}';E+E'\right)\quad \text{),}
\end{align*}
we obtain, by Corollary~\ref{cor:YuanMoriwaki},
\begin{align*}
&\avolq{X|Y}\left(\overline{D}+\overline{D}';E+E'\right)^{1/(\dim Y+1)} \\
&\qquad\qquad \geq\avol_{X|Y}\left(\overline{D};E\right)^{1/(\dim Y+1)}+\avol_{X|Y}\left(\overline{D}';E'\right)^{1/(\dim Y+1)}-3\varepsilon \\
\text{(respectively,}\quad &\avolq{\quot(X|Y)}\left(\overline{D}+\overline{D}';E+E'\right)^{1/(\dim Y+1)} \\
& \geq\avol_{\quot(X|Y)}\left(\overline{D};E\right)^{1/(\dim Y+1)}+\avol_{\quot(X|Y)}\left(\overline{D}';E'\right)^{1/(\dim Y+1)}-3\varepsilon\quad\text{).}
\end{align*}

(4): Since the cone
\[
\left\{\left(\overline{D};E\right)\in\aDDiv_{\RR,\RR}(X)\,\colon \,\text{$\left(\overline{D};E\right)$ is $Y$-big and $\ord_Y(E)>0$}\right\}
\]
is open in $\aDDiv_{\RR,\RR}(X)$ (see \cite[Theorem~2.21(2)]{IkomaDiff1}), the case (a) follows from the assertion (3) and \cite[Theorem~5.2]{Ein_Laz_Mus_Nak_Pop06} (see also \cite[Proposition~1.3.1]{MoriwakiEst}).

To show the case (b), we may assume without loss of generality that $B_j$'s are all effective.
We endow each $B_j$ with an adelic $B_j$-Green function such that $\overline{B}_j$ is effective.
The assertion follows from the estimate
\begin{align*}
\avolq{?}\left(\overline{D}+\pmb{\varepsilon}\cdot\overline{\bm{A}}-\lvert\pmb{\delta}\rvert\cdot\overline{\bm{B}};E\right) &\leq\avolq{?}\left(\overline{D}+\pmb{\varepsilon}\cdot\overline{\bm{A}};E+\pmb{\delta}\cdot\bm{B}\right) \\
&\leq\avolq{?}\left(\overline{D}+\pmb{\varepsilon}\cdot\overline{\bm{A}}+\lvert\pmb{\delta}\rvert\cdot\overline{\bm{B}};E\right)
\end{align*}
(see Notation and terminology~\ref{NC:norm}).
\end{proof}

\begin{proposition}\label{prop: arv rescaling}
Let $\left(\overline{D};E\right)\in\aDDiv_{\QQ,\QQ}(X)$ be a $Y$-big pair on $X$, let $S\subset M_K$ be a finite subset, and let $\pmb{\gamma}\coloneqq (\gamma_v)_{v\in S}$ be a family of real numbers.
Then
\begin{align*}
&\left|\avol_{\quot(X|Y)}\left(\overline{D}+\left(0,\sum_{v\in S}\gamma_v[v]\right);E\right)-\avol_{\quot(X|Y)}\left(\overline{D};E\right)\right| \\
&\qquad\qquad\qquad\qquad\qquad \leq (\dim X+1)[K:\QQ]\vol_{X|Y}(D-E)\|\pmb{\gamma}\|.
\end{align*}
\end{proposition}

\begin{proof}
See \cite[Proposition~5.1.2]{MoriwakiAdelic}.
\end{proof}

\subsection{Arithmetic restricted positive intersection numbers}\label{subsec: arithmetic restricted positive intersection number}

Let $X$ be a normal, projective, and geometrically connected $K$-variety, let $Y$ be a closed subvariety of $X$, and let $\left(\overline{D};E\right)\in\aDDiv_{\RR,\RR}(X)$ be a $Y$-big pair.
A \emph{$Y$-approximation of $\left(\overline{D};E\right)$} is defined as a couple $\left(\pi\colon X'\to X,\overline{M}\right)$ consisting of a birational $K$-morphism of projective varieties $\pi\colon X'\to X$ and a nef adelic $\RR$-Cartier divisor $\overline{M}$ on $X'$ having the following properties:
\begin{enumerate}
\renewcommand{\labelenumi}{(\alph{enumi})}
\item $X'$ is smooth and $\pi$ is isomorphic around the generic point of $Y$.
\item Let $\pi_*^{-1}(Y)$ denote the strict transform of $Y$ via $\pi$.
Then $\overline{M}$ is $\pi_*^{-1}(Y)$-big and $\left(\pi^*\overline{D}-\overline{M};E\right)$ is $\pi_*^{-1}(Y)$-pseudo-effective.
\end{enumerate}
We denote the set of all $Y$-approximations of $\left(\overline{D};E\right)$ by $\widehat{\Theta}_Y\left(\overline{D};E\right)$.
Moreover, we set
\begin{align*}
&\widehat{\Theta}_Y^{\rm rw}\left(\overline{D};E\right) \nonumber\\
&\quad \coloneqq \left\{(\pi,\overline{M})\in\widehat{\Theta}_Y\left(\overline{D};E\right)\,\colon \,\begin{array}{l}\text{$\overline{M}\in\aDiv_{\QQ}(X')$, $\overline{M}$ is w-ample,} \\ \text{and $\left(\pi^*\overline{D}-\overline{M};E\right)$ is $\pi_*^{-1}(Y)$-big}\end{array}\right\}.
\end{align*}

We define the \emph{arithmetic restricted positive intersection number of $\left(\overline{D};E\right)$ along $Y$} as
\begin{equation}
\left.\left\langle\left(\overline{D};E\right)^{\cdot(\dim Y+1)}\right\rangle\right|_Y\coloneqq \sup_{(\pi,\overline{M})\in\aTheta_Y\left(\overline{D};E\right)}\left\{\adeg\left(\left(\overline{M}|_{\pi_*^{-1}(Y)}\right)^{\cdot(\dim Y+1)}\right)\right\},
\end{equation}
where the restriction of $\overline{M}$ is defined up to arithmetic $\RR$-linear equivalence.

\begin{remark}\label{rem:arpin}
\begin{enumerate}
\item An arithmetic restricted positive intersection number is actually given as a limit as in \cite[Proposition~4.4]{IkomaCon}.
\item\label{item: rw approximation arpin} The same arguments as in \cite[Proposition~3.9]{IkomaDiff1} will lead to
\[
\left.\left\langle\left(\overline{D};E\right)^{\cdot(\dim Y+1)}\right\rangle\right|_Y=\sup_{(\pi,\overline{M})\in\widehat{\Theta}_Y^{\rm rw}\left(\overline{D};E\right)}\left\{\adeg\left(\left(\overline{M}|_{\pi_*^{-1}(Y)}\right)^{\cdot(\dim Y+1)}\right)\right\}.
\]
\item If $\left(\overline{D}_1;E_1\right)\preceq_Y\left(\overline{D}_2;E_2\right)$, then
\[
\left.\left\langle\left(\overline{D}_1;E_1\right)^{\cdot(\dim Y+1)}\right\rangle\right|_Y\leq\left.\left\langle\left(\overline{D}_2;E_2\right)^{\cdot(\dim Y+1)}\right\rangle\right|_Y.
\]
\item\label{item: arpin homogeneity} For any $a\in\RR_{\geq 0}$, one has
\[
\left.\left\langle\left(a\overline{D};aE\right)^{\cdot(\dim Y+1)}\right\rangle\right|_Y=a^{\dim Y+1}\left.\left\langle\left(\overline{D};E\right)^{\cdot(\dim Y+1)}\right\rangle\right|_Y.
\]
\item\label{item: remark arpin} The arithmetic restricted positive intersection numbers fit into the Brunn--Minkowski inequality:
\begin{align*}
&\left.\left\langle\left(\overline{D}_1;E_1\right)^{\cdot(\dim Y+1)}\right\rangle\right|_Y^{1/(\dim Y+1)}+\left.\left\langle\left(\overline{D}_2;E_2\right)^{\cdot(\dim Y+1)}\right\rangle\right|_Y^{1/(\dim Y+1)} \\
&\qquad\qquad\qquad\qquad\qquad\qquad \leq \left.\left\langle\left(\overline{D}_1+\overline{D}_2;E_1+E_2\right)^{\cdot(\dim Y+1)}\right\rangle\right|_Y^{1/(\dim Y+1)}
\end{align*}
for any $Y$-big pairs $\left(\overline{D}_1;E_1\right),\left(\overline{D}_2;E_2\right)\in\aDDiv_{\RR,\RR}(X)$ (see for example \cite[Theorem 2.9(4)]{IkomaCon}).
\item\label{item: continuity arpin} Assume that $Y$ has codimension one in $X$.
Let $\left(\overline{D};E\right)\in\aDDiv_{\RR,\RR}(X)$ be a $Y$-big pair, let $\overline{\bm{D}}\coloneqq\left(\overline{D}_1,\dots,\overline{D}_m\right)\in\aDiv_{\RR}(X)^{\times m}$, and let $\bm{E}\coloneqq\left(E_1,\dots,E_n\right)\in\Div_{\RR}(X)^{\times n}$.
Assume that one of the following two conditions is satisfied.
\begin{enumerate}
\item[(a)] $\ord_Y(E)>0$.
\item[(b)] $\ord_Y(E_j)=0$ for every $j$.
\end{enumerate}
Then one has
\[
\lim_{\pmb{\varepsilon},\,\pmb{\delta}\to 0}\left\langle\left(\overline{D}+\pmb{\varepsilon}\cdot\overline{\bm{D}};E+\pmb{\delta}\cdot\bm{E}\right)^{\cdot\dim X}\right\rangle\Biggm|_Y=\left.\left\langle\left(\overline{D};E\right)^{\cdot\dim X}\right\rangle\right|_Y
\]
(see Corollary~\ref{cor:Yuan_main} (4)).
\end{enumerate}
\end{remark}

\begin{proposition}\label{prop: arithmetic Fujita approximation}
If $\left(\overline{D};E\right)\in\aDDiv_{\QQ,\QQ}(X)$ is a $Y$-big pair on $X$ and $E\geq 0$, then
\[
\avolq{X|Y}\left(\overline{D};E\right)=\left.\left\langle\left(\overline{D};E\right)^{\cdot(\dim Y+1)}\right\rangle\right|_Y.
\]
\end{proposition}

\begin{proof}
The proof is almost the same as \cite[Proof of Theorem~8.4]{IkomaRem}, so we are going to only outline it.
Obviously, it suffices to show the inequality $\leq$.
By homogeneity (see Corollary~\ref{cor:Yuan_main} \eqref{item: cor Yuan main homogeneity} and Remark~\ref{rem:arpin} \eqref{item: arpin homogeneity}), we can assume $\left(\overline{D};E\right)\in\aDDiv_{\ZZ,\ZZ}(X)$.
We regard $\widehat{\Gamma}^{\rm ss}\left(m\overline{D};mE\right)\subset H^0\left(\mathcal{O}_X\left(m\left(D-E\right)\right)\right)$, and set
\begin{equation}
 \widehat{\mathfrak{b}}_m\coloneqq \Image\left(\left\langle\widehat{\Gamma}^{\rm ss}\left(m\overline{D};mE\right)\right\rangle_K\otimes_K\mathcal{O}_X\left(-m\left(D-E\right)\right)\to\mathcal{O}_X\right)
\end{equation}
for each $m\in\ZZ_{\geq 1}$.
Let $\pi_{m,K}\colon X_m\to X$ be a desingularization of the blow-up along $\widehat{\mathfrak{b}}_m$ (see \cite{Hironaka64}), let $Y_m\coloneqq \pi_{m,K*}^{-1}(Y)$ denote the strict transform of $Y$ via $\pi_{m,K}$, and let
\[
\mathcal{O}_{X_m}(F_m)\coloneqq \mathcal{H}om_{\mathcal{O}_{X_m}}\left(\widehat{\mathfrak{b}}_m\mathcal{O}_{X_m},\mathcal{O}_{X_m}\right).
\]
We define an adelic $F_m$-Green function $\pmb{g}^{\overline{F}_m}$ by
\[
g_v^{\overline{F}_m}(x)\coloneqq \min_{\phi\in\widehat{\Gamma}^{\rm ss}\left(m\overline{D};mE\right)\setminus\{0\}}\left\{g_v^{m\overline{D}+\widehat{(\phi)}}\left(\pi_{m,K,v}^{\rm an}(x)\right)\right\}
\]
for $v\in M_K$ and $x\in X_{m,v}^{\rm an}$.
Then $\left(\overline{F}_m;mE\right)$ is effective and $\overline{M}_m\coloneqq \pi_{m,K}^*\left(m\overline{D}\right)-\overline{F}_m$ is nef.
Moreover, by the same arguments as in \cite[Proposition~4.7]{IkomaRem}, we have
\[
\overline{M}_m\leq_{Y_m}\left(\pi_{m,K}^*\left(m\overline{D}\right);mE\right),
\]
and the natural images of $\widehat{\Gamma}^{\rm ss}_{X_m|Y_m}\left(\overline{M}_m\right)$ and of $\widehat{\Gamma}^{\rm ss}_{X|Y}\left(m\overline{D};mE\right)$ coincide in $H^0\left(\pi_{m,K}^*\left(m\left(D-E\right)\right)\right)$ for every $m\in\ZZ_{\geq 1}$.

By Fekete's lemma, the sequence
\[
 \left(\frac{\avolq{X_m|Y_m}\left(\overline{M}_m\right)}{m^{\dim Y+1}}\right)_{m\in\ZZ_{\geq 1}}
\]
converges (see \cite[Claim~8.6]{IkomaRem}).
We choose a normal and projective $O_K$-model $\mathscr{X}$ of $X$ (respectively, $\mathscr{X}_m$ of $X_m$) having the following properties:
\begin{enumerate}
\renewcommand{\labelenumi}{(\alph{enumi})}
\item The Zariski closure $\mathscr{Y}$ (respectively, $\mathscr{Y}_m$) of $Y$ in $\mathscr{X}$ (respectively, in $\mathscr{X}_m$) is Cartier.
\item There exists an arithmetic Cartier divisor $\overline{\mathscr{M}}$ on $\mathscr{X}$ such that
\[
\widehat{\Gamma}^{\rm f}_{X|Y}\left(m\overline{D};mE\right)\subset H^0\left(\left.\mathcal{O}_{\mathscr{X}}\left(m\mathscr{M}\right)\right|_{\mathscr{Y}}\right)
\]
and
\[
\widehat{\Gamma}^{\rm ss}_{X|Y}\left(m\overline{D};mE\right)\subset\widehat{\Gamma}^{\rm s}\left(\left.\mathcal{O}_{\mathscr{X}}\left(m\overline{\mathscr{M}}\right)\right|_{\mathscr{Y}}\right)
\]
for every $m\in\ZZ_{\geq 1}$.
\item There exists a projective and birational morphism $\pi_m\colon \mathscr{X}_m\to\mathscr{X}$ extending $\pi_{m,K}$.
\item There exists a Zariski closed subset $\mathscr{Z}$ of $\mathscr{X}$ such that $\pi_m$ is isomorphic over $\mathscr{X}\setminus\mathscr{Z}$ for every $m\in\ZZ_{\geq 1}$.
\end{enumerate}
Let $\pi_m'\colon \mathscr{Y}_m\to\mathscr{Y}$ be the morphism induced from $\pi_m$.
Given any $\varepsilon\in\RR_{>0}$, we can take a prime number $p$ such that
\begin{equation}\label{eqn:genFujita1}
 \frac{C'\left(\left.\overline{\mathscr{M}}^{\rm ad}\right|_Y,Y\right)}{\log(p)}\leq\varepsilon
\end{equation}
and such that there exists a good flag
\[
\mathscr{F}_{\geq 1}\colon \mathscr{Y}=\mathscr{F}_1\supset\mathscr{F}_2\supset\dots\supset\mathscr{F}_{\dim\mathscr{X}}
\]
on $\mathscr{Y}$ over $p$ such that $\mathscr{F}_{\dim\mathscr{X}}$ is not contained in $\mathscr{Z}$ (see \cite[Lemma~6.4]{IkomaRem}).
Let
\[
{\pi}_{m*}^{\prime -1}\left(\mathscr{F}_{\geq 1}\right)\colon \mathscr{Y}_m\supset\pi_{m*}^{\prime -1}\left(\mathscr{F}_2\right)\supset\dots\supset\pi_{m*}^{\prime -1}\left(\mathscr{F}_{\dim\mathscr{X}}\right)
\]
denote the flag on $\mathscr{Y}_m$ obtained by taking the strict transforms of $\mathscr{F}_{\geq 1}$ (see \cite[Lemma~6.3(2)]{IkomaRem}), and set
\[
 \Delta(m)\coloneqq \overline{\left(\bigcup_{k\geq 1}\frac{1}{km}\bm{w}_{\pi_{m*}^{\prime -1}(\mathscr{F}_{\geq 1})}\left(\widehat{\Gamma}_{CL(X_m|Y_m)}^{\rm ss}\left(k\overline{M}_m\right)\setminus\{0\}\right)\right)}
\]
for every sufficiently large $m\in\ZZ_{\geq 1}$.
Then, by \cite[Th{\'e}or{\`e}me~1.15]{BoucksomBourbaki}, we have
\[
 \vol_{\RR^{\dim Y+1}}\left(\Delta(m)\right)\log(p)\geq\vol_{\RR^{\dim Y+1}}\left(\widehat{\Delta}_{CL(X|Y)}^{\mathscr{F}_{\geq 1}}\left(\overline{D};E\right)\right)\log(p)-\varepsilon
\]
for every sufficiently large $m$.
On the other hand, by Corollary~\ref{cor:YuanMoriwaki}, we have
\[
 \vol_{\RR^{\dim Y+1}}\left(\widehat{\Delta}_{CL(X|Y)}^{\mathscr{F}_{\geq 1}}\left(\overline{D};E\right)\right)\log(p)\geq\frac{\avolq{X|Y}\left(\overline{D};E\right)}{(\dim Y+1)!}-\varepsilon
\]
and
\[
 \frac{\avolq{X_m|Y_m}\left(\overline{M}_m\right)}{(\dim Y+1)!m^{\dim Y+1}}\geq\vol_{\RR^{\dim Y+1}}\left(\Delta(m)\right)\log(p)-\varepsilon.
\]
Thus
\[
 \lim_{m\to\infty}\frac{\avolq{X_m|Y_m}\left(\overline{M}_m\right)}{m^{\dim Y+1}}\geq\avolq{X|Y}\left(\overline{D};E\right)-3(\dim Y+1)!\varepsilon.
\]
Hence, by \cite[Proposition~8.1]{IkomaRem}, we have $\left.\left\langle\left(\overline{D};E\right)^{\cdot(\dim Y+1)}\right\rangle\right|_Y\geq\avolq{X|Y}\left(\overline{D};E\right)$.
\end{proof}

\subsection{Proof of Theorem~B}\label{subsec:Cont_at_the_boundary}

Let $\mathscr{X}$ be a normal, projective, and arithmetic $O_K$-variety with smooth generic fiber $X$.
Let $\overline{\mathscr{A}}$ be an arithmetic Cartier divisors of $C^{\infty}$-type on $\mathscr{X}$ such that $\mathscr{A}$ is ample, $\overline{\mathscr{A}}^{\rm ad}$ is w-ample, and the curvature form $c_1\left(\overline{\mathscr{A}}\right)$ is positive pointwise on $X_{\infty}^{\rm an}$.
Let $\overline{\pmb{\mathscr{E}}}\coloneqq\left(\overline{\mathscr{E}}_1,\dots,\overline{\mathscr{E}}_l\right)$ be a family of horizontal arithmetic Cartier divisors of $C^{\infty}$-type on $\mathscr{X}$.

By Lemma~\ref{lem: Estimates II generation}, one can find a $p_0\in\ZZ_{\geq 1}$ and an $a_0\in\QQ_{>0}$ such that, for every $(p,\bm{q})\in(\ZZ_{\geq 0})^{l+1}$ with $p\geq p_0$ and $\|\bm{q}\|\leq a_0p$, the following three conditions are satisfied:
\begin{enumerate}
\renewcommand{\labelenumi}{(\alph{enumi})}
\item $\left(p\overline{\mathscr{A}}-\bm{q}\cdot\overline{\pmb{\mathscr{E}}}\right)^{\rm ad}$ is ample with pointwise positive curvature form.
\item The evaluation map
\begin{equation}\label{eqn:evaluation_one-side}
H^0\left(\left.\left(p\mathscr{A}-\bm{q}\cdot\pmb{\mathscr{E}}\right)\right|_X\right)\otimes_K\mathcal{O}_X\to\mathcal{O}_X\left(\left.\left(p\mathscr{A}-\bm{q}\cdot\pmb{\mathscr{E}}\right)\right|_X\right)
\end{equation}
is surjective.
\item $H^0(p\mathscr{A}-\bm{q}\cdot\pmb{\mathscr{E}})$ is generated by the strictly small sections over $\QQ$: namely
\begin{equation}\label{eqn:generation_one-side}
H^0(p\mathscr{A}-\bm{q}\cdot\pmb{\mathscr{E}})\otimes_{\ZZ}\QQ=\left\langle\widehat{\Gamma}^{\rm ss}\left(p\overline{\mathscr{A}};\bm{q}\cdot\pmb{\mathscr{E}}\right)\right\rangle_{\QQ}.
\end{equation}
\end{enumerate}

Given any $(p,\bm{q})\in(\ZZ_{\geq 0})^{l+1}$ with $p\geq p_0$ and $\|\bm{q}\|\leq a_0p$, we endow $\bm{q}\cdot\pmb{\mathscr{E}}$ with a Green function as
\[
\overline{\bm{q}\cdot\pmb{\mathscr{E}}}^{(p,\bm{q})}\coloneqq \left(\bm{q}\cdot\pmb{\mathscr{E}},\min_{\phi\in\widehat{\Gamma}^{\rm ss}\left(p\overline{\mathscr{A}};\bm{q}\cdot\pmb{\mathscr{E}}\right)\setminus\{0\}}\left\{g^{p\overline{\mathscr{A}}+\widehat{(\phi)}}\right\}\right),
\]
which is an effective arithmetic Cartier divisor on $\mathscr{X}$ having the following two properties (see Proof of Proposition~\ref{prop: arithmetic Fujita approximation}):
\begin{enumerate}
\renewcommand{\labelenumi}{(\alph{enumi})}
\item $p\overline{\mathscr{A}}-\overline{\bm{q}\cdot\pmb{\mathscr{E}}}^{(p,\bm{q})}$ is of $(\textup{PSH}\cap C^0)$-type (see \cite[section~2.3]{MoriwakiZar}).
\item $\widehat{\Gamma}^{\rm ss}_{\mathscr{X}|\mathscr{Y}}\left(p\overline{\mathscr{A}};\bm{q}\cdot\pmb{\mathscr{E}}\right)\subset\widehat{\Gamma}^{\rm ss}_{\mathscr{X}|\mathscr{Y}}\left(p\overline{\mathscr{A}}-\overline{\bm{q}\cdot\pmb{\mathscr{E}}}^{(p,\bm{q})}\right)$.
\end{enumerate}

\begin{lemma}\label{lem:Ess_Inclusion}
Let $p_0\in\ZZ_{\geq 1}$ and $a_0\in\QQ_{>0}$ be as above.
Fix any $\varepsilon\in\RR$ with $0<\varepsilon\leq 1$, and let $\lambda_{\varepsilon}\in\RR_{>0}$ be as in Lemma~\ref{lem:PSH_Env}.
Then one has
\[
\widehat{\Gamma}^{\rm ss}_{\mathscr{X}|\mathscr{Y}}\left(p\overline{\mathscr{A}};\bm{q}\cdot\pmb{\mathscr{E}}\right)\subset\widehat{\Gamma}^{\rm ss}\left(\left.\left(p\left(\overline{\mathscr{A}}(\varepsilon)\right)-\bm{q}\cdot\left(\overline{\pmb{\mathscr{E}}}(-\lambda_{\varepsilon})\right)\right)\right|_{\mathscr{Y}}\right)
\]
for any $(p,\bm{q})\in(\ZZ_{\geq 0})^{l+1}$ with $p\geq p_0$ and $\|\bm{q}\|\leq a_0p$.
\end{lemma}

\begin{proof}
By Lemma~\ref{lem:PSH_Env}, we have $g^{p\overline{\mathscr{A}}-\overline{\bm{q}\cdot\pmb{\mathscr{E}}}^{(p,\bm{q})}}\leq g^{p\overline{\mathscr{A}}-\bm{q}\cdot\overline{\pmb{\mathscr{E}}}}+p\varepsilon+\|\bm{q}\|\lambda_{\varepsilon}$.
Hence
\[
\widehat{\Gamma}^{\rm ss}_{\mathscr{X}|\mathscr{Y}}\left(p\overline{\mathscr{A}};\bm{q}\cdot\pmb{\mathscr{E}}\right)\subset\widehat{\Gamma}^{\rm ss}_{\mathscr{X}|\mathscr{Y}}\left(p\overline{\mathscr{A}}-\overline{\bm{q}\cdot\pmb{\mathscr{E}}}^{(p,\bm{q})}\right)\subset\widehat{\Gamma}^{\rm ss}\left(\left.\left(p\left(\overline{\mathscr{A}}(\varepsilon)\right)-\bm{q}\cdot\left(\overline{\pmb{\mathscr{E}}}(-\lambda_{\varepsilon})\right)\right)\right|_{\mathscr{Y}}\right).
\]
\end{proof}

\begin{proposition}\label{prop: Cont at Bound}
Let $\mathscr{X}$ be a normal, projective, and arithmetic $O_K$-variety with smooth generic fiber $X$, and let $Y$ be a prime Cartier divisor on $X$.
Let $\overline{\mathscr{A}}$ be an arithmetic Cartier divisor of $C^{\infty}$-type on $\mathscr{X}$ such that $\mathscr{A}$ is ample and such that the curvature form $c_1\left(\overline{\mathscr{A}}\right)$ is positive pointwise on $X_{\infty}^{\rm an}$.
Let $\overline{\pmb{\mathscr{D}}}\coloneqq\left(\overline{\mathscr{D}}_1,\dots,\overline{\mathscr{D}}_m\right)$ be a family of arithmetic Cartier divisors of $C^{\infty}$-type on $\mathscr{X}$, and let $\pmb{\mathscr{E}}=\left(\mathscr{E}_1,\dots,\mathscr{E}_n\right)$ be a family of effective horizontal Cartier divisor on $\mathscr{X}$.
Then we have
\[
\lim_{\substack{\bm{t}\to 0, \\ \bm{r}\downarrow 0}}\left.\left\langle\left(\left(\overline{\mathscr{A}}+\bm{t}\cdot\overline{\pmb{\mathscr{D}}};\bm{r}\cdot\pmb{\mathscr{E}}\right)^{\rm ad}\right)^{\cdot\dim X}\right\rangle\right|_Y\leq \left.\left\langle\left(\overline{\mathscr{A}}^{\rm ad}\right)^{\cdot\dim X}\right\rangle\right|_Y.
\]
\end{proposition}

\begin{proof}
Let $p_0\in\ZZ_{\geq 1}$ and $a_0\in\QQ_{>0}$ be as above.
We fix any $\varepsilon\in\RR_{>0}$ with $0<\varepsilon\leq 1$, and let $\lambda_{\varepsilon}\in\RR_{>0}$ be as in Lemma~\ref{lem:PSH_Env}.
We fix a $b_0\in\ZZ_{\geq 1}$ such that $\overline{\mathscr{D}}_i^{\rm ad}\preceq_Yb_0\overline{\mathscr{A}}^{\rm ad}$ for every $i$.
For each $j$, we fix an $\mathscr{E}_j$-Green function of $C^{\infty}$-type.
By Lemma~\ref{lem:Ess_Inclusion}, we have
\[
\widehat{\Gamma}_{\mathscr{X}|\mathscr{Y}}^{\rm ss}\left(p\overline{\mathscr{A}};\bm{q}\cdot\pmb{\mathscr{E}}\right)\subset\widehat{\Gamma}_{\mathscr{X}|\mathscr{Y}}^{\rm ss}\left(\left.\left(p\left(\overline{\mathscr{A}}(\varepsilon)\right)-\bm{q}\cdot\left(\overline{\pmb{\mathscr{E}}}(-\lambda_{\varepsilon})\right)\right)\right|_{\mathscr{Y}}\right)
\]
for every $(p,\bm{q})\in(\ZZ_{\geq 0})^{n+1}$ with $p\geq p_0$ and $\|\bm{q}\|\leq a_0 p$.
Hence
\[
\avol_{X|Y}\left(\left(\overline{\mathscr{A}}+\bm{t}\cdot\overline{\pmb{\mathscr{D}}};\bm{r}\cdot\pmb{\mathscr{E}}\right)^{\rm ad}\right)\leq\avol\left(\left.\left((1+\|\bm{t}\|b_0)\overline{\mathscr{A}}(\varepsilon)-\bm{r}\cdot\overline{\pmb{\mathscr{E}}}(-\lambda_{\varepsilon})\right)^{\rm ad}\right|_Y\right)
\]
for every $\bm{t}\in\QQ^m$ and $\bm{r}\in(\QQ_{\geq 0})^n$ with $\|\bm{r}\|\leq a_0$.
By taking $\bm{t}\to 0$ and $\bm{r}\downarrow 0$,
\begin{align*}
&\lim_{\substack{\bm{t}\to 0, \\ \bm{r}\downarrow 0}}\left.\left\langle\left(\left(\overline{\mathscr{A}}+\bm{t}\cdot\overline{\pmb{\mathscr{D}}};\bm{r}\cdot\pmb{\mathscr{E}}\right)^{\rm ad}\right)^{\cdot\dim X}\right\rangle\right|_Y \leq\avol\left(\left.\left(\overline{\mathscr{A}}(\varepsilon)\right)^{\rm ad}\right|_Y\right) \\
&\qquad\qquad\qquad =\left.\left\langle\left(\overline{\mathscr{A}}^{\rm ad}\right)^{\cdot\dim X}\right\rangle\right|_Y+\varepsilon(\dim X)[K_Y:\QQ]\vol\left(\left.\mathscr{A}\right|_Y\right)
\end{align*}
(see \cite[Corollary~7.2 (1)]{MoriwakiEst}).
Hence we conclude the proof by taking $\varepsilon\downarrow 0$.
\end{proof}

\begin{proof}[Proof of Theorem~B]
We start proving Theorem~B.
We may assume without loss of generality that $E_j$'s are all effective.
Let $\pmb{g}^{\overline{Y}}$ (respectively, $\pmb{g}^{\overline{E}_j}$) be an adelic Green function such that $\overline{Y}\coloneqq\left(Y,\pmb{g}^{\overline{Y}}\right)$ is effective (respectively, $\overline{E}_j\coloneqq\left(E_i,\pmb{g}^{\overline{E}_j}\right)$ is effective for every $j$).
Then 
\begin{align*}
&\lim_{\substack{\bm{t},\bm{u}\to 0, \\ r\downarrow 0}}\left.\left\langle\left(\overline{A}+\bm{t}\cdot\overline{\bm{D}};rY+\bm{u}\cdot\bm{E}\right)^{\cdot\dim X}\right\rangle\right|_Y \\
&\qquad\qquad\qquad \geq \lim_{\substack{\bm{t},\bm{u}\to 0, \\ r\downarrow 0}}\left.\left\langle\left(\overline{A}+\bm{t}\cdot\overline{\bm{D}}-r\overline{Y}-|\bm{u}|\cdot\overline{\bm{E}}\right)^{\cdot\dim X}\right\rangle\right|_Y=\left.\left\langle\overline{A}^{\cdot\dim X}\right\rangle\right|_Y.
\end{align*}
Hence it suffices to show the reverse inequality.

First, we treat the case where $\overline{A}$, $\overline{\bm{D}}$, and $\bm{E}$ are all integral and $\overline{A}$ is ample.
Let $\varphi:X_1\to X$ be a resolution of singularities of $X$ (see \cite{Hironaka64}), and let $\varphi_*^{-1}:\Div_{\RR}(X)\to\Div_{\RR}(X_1)$ denote the strict transform via $\varphi$.
Then
\[
\left.\left\langle\left(\overline{A}+\bm{t}\cdot\overline{\bm{D}};rY+\bm{u}\cdot\bm{E}\right)^{\cdot\dim X}\right\rangle\right|_Y=\left.\left\langle\left(\varphi^*\left(\overline{A}+\bm{t}\cdot\overline{\bm{D}}\right);rY+\bm{u}\cdot\bm{E}\right)^{\cdot\dim X}\right\rangle\right|_{\varphi_*^{-1}(Y)}
\]
for every $r\in\QQ$, $\bm{t}\in\QQ^m$, and $\bm{u}\in\QQ^n$ (see \cite[Lemma~2.15]{IkomaDiff1}).
Hence we may assume that $X$ is smooth.
We can take a finite subset $S\subset M_K$ containing $\{\infty\}$ such that $\left(g_v^{\overline{A}}\right)_{v\in M_K\setminus S}$ is defined on a suitable $O_K$-model of $X$.
We fix any sufficiently small $\gamma_0\in\RR$ with $0<\gamma_0<1$ such that
\[
\overline{A}\left(-\pmb{\gamma}^S\right)\coloneqq \overline{A}-\left(0,\sum_{v\in S}\gamma[v]\right)
\]
is w-ample for every $\gamma\in\RR$ with $0\leq \gamma\leq \gamma_0$.
Given such a $\gamma\in\RR$, we can find, by using \cite[Proposition~4.4.2]{MoriwakiAdelic} and \cite[Theorem~4.6]{MoriwakiZar}, an $O_K$-model $\left(\mathscr{X}_{\gamma},\overline{\mathscr{B}_{\gamma}}\right)$ consisting of a normal and projective $O_K$-model $\mathscr{X}_{\gamma}$ of $X$ and an arithmetic $\QQ$-Cartier divisor $\overline{\mathscr{B}_{\gamma}}$ of $C^{\infty}$-type on $\mathscr{X}_{\gamma}$ having the following properties (see also \cite[Proposition~3.1]{IkomaCon}):
\begin{enumerate}
\renewcommand{\labelenumi}{(\alph{enumi})}
\item $\mathscr{B}_{\gamma}$ is ample.
\item The curvature form $c_1\left(\overline{\mathscr{B}_{\gamma}}\right)$ is positive pointwise on $X_{\gamma,\infty}^{\rm an}$.
\item $(1-\gamma)\overline{A}\left(-\pmb{\gamma}^S\right)\preceq_Y\overline{\mathscr{B}_{\gamma}}^{\rm ad}\preceq_Y\overline{A}$.
\item The Zariski closure $\mathscr{Y}_{\gamma}$ of $Y$ in $\mathscr{X}_{\gamma}$ is Cartier.
\end{enumerate}
Let $\alpha_0$ be an integer such that $\alpha_0\overline{A}-\overline{D}_i$ are w-ample for all $i$ (see \cite[Lemma~5.3]{IkomaRem}).
Applying Proposition~\ref{prop: Estimates II inclusions} to $\overline{A}\left(-\pmb{\gamma}_0^S\right)$, we can find a $\rho_0\in\QQ_{>0}$ such that
\begin{align*}
&\left.\left\langle\left(\overline{A}\left(-\pmb{\gamma}^S\right)+\alpha_0\|\bm{t}\|\overline{A};rY+\bm{u}\cdot\bm{E}\right)^{\cdot\dim X}\right\rangle\right|_Y \\
&\qquad\qquad\qquad\qquad\qquad =\avol_{\quot(X|Y)}\left(\overline{A}\left(-\pmb{\gamma}^S\right)+\alpha_0\|\bm{t}\|\overline{A};rY+\bm{u}\cdot\bm{E}\right)
\end{align*}
for every $\gamma\in\RR$, $r\in\QQ_{\geq 0}$, $\bm{t}\in\QQ^m$, and $\bm{u}\in\QQ^n$ with $0\leq\gamma\leq\gamma_0$ and $r+\|\bm{u}\|\leq\rho_0$.
By Remark~\ref{rem:arpin} \eqref{item: arpin homogeneity} and Proposition~\ref{prop: arv rescaling},
\begin{multline*}
(1-\gamma)^{\dim X}\Biggl(\left.\left\langle\left(\left(1+\alpha_0\|\bm{t}\|\right)\overline{A};rY+\bm{u}\cdot\bm{E}\right)^{\cdot\dim X}\right\rangle\right|_Y \\
-(\dim X+1)[K:\QQ]\vol_{X|Y}\left((1+\alpha_0\|\bm{t}\|)A-rY-\bm{u}\cdot\bm{E}\right)\#S\gamma\Biggr) \\
\leq \left.\left\langle\left(\left(\left(1+\alpha_0\|\bm{t}\|\right)\overline{\mathscr{B}_{\gamma}};(1-\gamma)r\mathscr{Y}_{\gamma}\right)^{\rm ad}\right)^{\cdot\dim X}\right\rangle\right|_Y.
\end{multline*}
Taking $r\downarrow 0$, $\bm{t}\to 0$, $\bm{u}\to 0$, and $\gamma\downarrow 0$, we obtain the theorem by Proposition~\ref{prop: Cont at Bound}.

Next, we treat the general case.
By the previous arguments, we obtain
\begin{align*}
&\lim_{\substack{\bm{t},\bm{u}\to 0, \\ r\downarrow 0}}\left.\left\langle\left(\overline{A}+\bm{t}\cdot\overline{\bm{D}};rY+\bm{u}\cdot\bm{E}\right)^{\cdot\dim X}\right\rangle\right|_Y \\
&\qquad\qquad\qquad \leq \lim_{\substack{\bm{t},\bm{u}\to 0, \\ r\downarrow 0}}\left.\left\langle\left(\overline{A}'+\bm{t}\cdot\overline{\bm{D}};rY+\bm{u}\cdot\bm{E}\right)^{\cdot\dim X}\right\rangle\right|_Y=\left.\left\langle{\overline{A}'}^{\cdot\dim X}\right\rangle\right|_Y
\end{align*}
for every ample adelic $\QQ$-Cartier divisor $\overline{A}'$ such that $\overline{A}\preceq_Y\overline{A}'$.
Hence the theorem follows from continuity (see Remark~\ref{rem:arpin} \eqref{item: continuity arpin}).
\end{proof}

%%%
\section{Proof of Theorem~A}

In this section, we shall give an upper and a lower bounds for the one-sided directional derivatives of the arithmetic volume function along the directions defined by prime Cartier divisors (see Corollary~\ref{cor:upper_bound} and Theorem~\ref{thm:FE3}, respectively).
Theorem~A is a direct consequence of these two estimates.

\subsection{Differentiability of concave functions}\label{subsec: concave}

\begin{definition}
\begin{enumerate}
\item For $r\in\RR_{>0}$ and $v\in\RR^n$, we set
\[
B_r(v)\coloneqq \left\{p\in\RR^n\,\colon \,\|p-v\|<r\right\}
\]
(see Notation and terminology~\ref{NC:norm}).
Let $C$ be a nonempty open subset of $\RR^n$ and let $\KK$ denote either $\QQ$ or $\RR$.
A function $f\colon C\cap\KK^n\to\RR$ is said to be \emph{locally Lipschitz-continuous} on $C$ if, given any $a\in C$, there exist an $\varepsilon\in\RR_{>0}$ and an $L\in\RR_{>0}$ such that $B_{\varepsilon}(a)\subset C$ and
\[
|f(p)-f(q)|\leq L\|p-q\|
\]
for all $p,q\in B_{\varepsilon}(a)\cap\KK^n$.
\begin{enumerate}
\item As is well known, any concave function defined on $C\cap\QQ^n$ is locally Lipschitz-continuous on $C$ and extends uniquely to a continuous function defined on $C$ (see \cite[section 1.3]{MoriwakiEst}).
\item Suppose that $f,g\colon C\to\RR$ are locally Lipschitz-continuous functions on $C$.
The product $f\cdot g\colon C\to\RR$ is also locally Lipschitz-continuous on $C$.
If $g(p)\neq 0$ for all $p\in C$, then the quotient $f/g\colon C\to\RR$ is also locally Lipschitz-continuous on $C$.
\end{enumerate}
\item Let $e_i=(0,\dots,0,\overset{i}{1},0,\dots,0)$ denote the $i$-th standard basis vector of $\RR^n$, and let $f$ be a concave function defined on a nonempty convex open subset $C$ of $\RR^n$.
Then the $i$-th right (respectively, left) partial derivative of $f$ exists and is denoted by
\begin{align*}
&f_{x_i+}(p) \coloneqq \lim_{r\downarrow 0}\frac{f(p+re_i)-f(p)}{r}\\
\text{(respectively,}\quad &f_{x_i-}(p)\coloneqq \lim_{r\uparrow 0}\frac{f(p+re_i)-f(p)}{r}\text{)}
\end{align*}
for each $p\in C$.
\end{enumerate}
\end{definition}

\begin{lemma}\label{lem: fund property of concave function}
Let $C$ be a nonempty convex open subset of a Euclidean space $\RR^n$ and let $f\colon C\to\RR$ be a concave function.
Suppose that the function $f_{x_n+}\colon C\cap\QQ^n\to\RR$ is locally Lipschitz-continuous on $C$.
Then $f_{x_n}$ exists at any point in $C$ and $f_{x_n}\colon C\to\RR$ is continuous on $C$.
\end{lemma}

\begin{proof}
By hypothesis, there exists a unique continuous function $g$ on $C$ such that $f_{x_n+}=g$ on $C\cap\QQ^n$.
Denote a point in $C$ by $a=(a',b)$ with $a'\in\RR^{n-1}$ and $b\in\RR$.
We show the following claim:

\begin{claim}\label{clm: FP concave}
If $a'\in\QQ^{n-1}$, then $f_{x_n}$ exists at $a$ and $f_{x_n}(a)=g(a)$.
\end{claim}

\begin{proof}[Proof of Claim~\ref{clm: FP concave}]
Indeed, suppose that $f_{x_n}$ does not exist at $a$.
Then one has $f_{x_n+}(a)<f_{x_n-}(a)$ (see \cite[page 26, Theorem~2.7]{GruberBook}).
Since $f_{x_n+}\colon C\cap\QQ^n\to\RR$ is locally Lipschitz-continuous on $C$, there exist an $\varepsilon\in\RR_{>0}$ and an $L\in\RR_{>0}$ such that $B_{\varepsilon}(a)\subset C$ and
\begin{equation}\label{eqn: FP concave}
|f_{x_n+}(x)-f_{x_n+}(y)|\leq L\|x-y\|
\end{equation}
for all $x,y\in B_{\varepsilon}(a)\cap\QQ^n$.
Set $\varepsilon'\coloneqq \min\{\varepsilon,(f_{x_n-}(a)-f_{x_n+}(a))/2L\}$, and choose two rational numbers $p,q\in\QQ$ such that $b-\varepsilon'<q<b<p<b+\varepsilon'$.
Then
\[
f_{x_n+}(a',p)\leq f_{x_n+}(a)<f_{x_n-}(a)\leq f_{x_n+}(a',q),
\]
which contradicts the property \eqref{eqn: FP concave}.
The assertion that $f_{x_i}(a)=g(a)$ is then obvious (see for example \cite[page 27, Theorem~2.8]{GruberBook}).
\end{proof}

We set
\[
F(x',x_n)\coloneqq f(x',b)+\int_b^{x_n}g(x',r)\,dr,
\]
which is defined on a suitably small open neighborhood $U$ of $a$.
Since both $f$ and $F$ are continuous on $U$ and coincide on $U\cap(\QQ^{n-1}\times\RR)$, they are identical on $U$.
Hence $f_{x_n}$ exists at $a$ and $f_{x_n}(a)=F_{x_n}(a)=g(a)$.
\end{proof}

\subsection{Upper bound}\label{subsec: upper bound}

\begin{theorem}\label{thm:FE}
Let $\ast$ denote either ${\rm ss}$ or ${\rm s}$.
Let $X$ be a normal, projective, and geometrically connected $K$-variety and let $Y$ be any effective Cartier divisor on $X$.
We fix an adelic Cartier divisor $\overline{A}$ on $X$ such that $\widehat{\Gamma}^{\rm s}_{X|Y}\left(\overline{A}\right)\neq\{0\}$ and such that $\widehat{\Gamma}^{\rm s}_{X|Y}\left(\overline{A};Y\right)\neq\{0\}$.
Then, for any $\left(\overline{D};E\right)\in\aDDiv_{\ZZ,\ZZ}(X)$ with $E\geq 0$ and for any $n\in\ZZ_{\geq 0}$, we have
\begin{align*}
0 &\leq\ah\left(\overline{D};E\right)-\ah\left(\overline{D};E+nY\right) \\
&\qquad\qquad \leq n\ah_{X|Y}\left(\overline{D}(\log(2))+n\overline{A};E+nY\right)+\log(6)\rk_{\QQ} H^0(D)
\end{align*}
and
\begin{align*}
0 &\leq\ah\left(\overline{D};E-nY\right)-\ah\left(\overline{D};E\right) \\
&\qquad\qquad \leq n\ah_{X|Y}\left(\overline{D}(\log(2))+n\overline{A};E\right)+\log(6)\rk_{\QQ} H^0(D).
\end{align*}
\end{theorem}

\begin{proof}
Let $1_Y$ denote the canonical section of $\OO_X(Y)$.
By applying Remark~\ref{rem:exact_sequence} \eqref{item: exact sequences combined} to the exact sequence
\[
0\to H^0\left(D-E-nY\right)\xrightarrow{\otimes 1_Y^{\otimes n}} H^0\left(D-E\right) \to H^0_{X|nY}\left(D-E\right)\to 0,
\]
one obtains
\begin{equation}
\ah\left(\overline{D};E\right)-\ah\left(\overline{D};E+nY\right)\leq\ah_{X|nY}\left(\overline{D};E\right)+\log(6)\rk_{\QQ} H^0(D). \label{eqn:FE_Main1}
\end{equation}
We are going to estimate the term $\ah_{X|nY}\left(\overline{D};E\right)$.
Applying Lemma~\ref{lem:quot_exact} to the diagram
{\footnotesize
\[
\xymatrix{
 0 \ar[r] & H^0_{X|Y}\left(D-E-kY\right) \ar[r] & H^0_{X|(k+1)Y}\left(D-E\right) \ar[r] & H^0_{X|kY}\left(D-E\right) \ar[r] & 0 \\
 0 \ar[r] & H^0\left(D-E-kY\right) \ar[r]^-{\otimes 1_Y^{\otimes k}} \ar[u] & H^0\left(D-E\right) \ar[r] \ar[u] & H^0_{X|kY}\left(D-E\right) \ar[r] \ar@{=}[u] & 0,
}
\]}%
one has
\begin{equation}
\ah_{X|(k+1)Y}\left(\overline{D};E\right)-\ah_{X|kY}\left(\overline{D};E\right)\leq\ah_{X|Y}\left(\overline{D}(\log(2));E+kY\right)\label{eqn:FE_Main2}
\end{equation}
for each $k$.
By adding \eqref{eqn:FE_Main2} for $k=1,2,\dots,n-1$, one obtains
\begin{align*}
\ah_{X|nY}\left(\overline{D};E\right) &\leq\sum_{k=0}^{n-1}\ah_{X|Y}\left(\overline{D}(\log(2));E+kY\right) \\
&\leq n\ah_{X|Y}\left(\overline{D}(\log(2))+n\overline{A};E+nY\right)
\end{align*}
as required.
To show the second inequality, we may assume $0\leq n\leq\ord_Y(E)$.
It also follows from the same arguments as above by replacing $E$ with $E-nY$.
\end{proof}

\begin{corollary}\label{cor:upper_bound}
Let $X$ be a normal, projective, and geometrically connected $K$-variety, let $Y$ be a prime Cartier divisor on $X$, and let $\left(\overline{D};E\right)\in\aDDiv_{\QQ,\QQ}(X)$ be a $Y$-big pair.
If $\ord_Y(E)>0$, then
\[
\lim_{r\downarrow 0}\frac{\avol\left(\overline{D};E\right)-\avol\left(\overline{D};E+rY\right)}{r}\leq(\dim X+1)\left.\left\langle\left(\overline{D};E\right)^{\cdot\dim X}\right\rangle\right|_Y.
\]
\end{corollary}

\begin{proof}
One may assume $\left(\overline{D};E\right)\in\aDDiv_{\ZZ,\ZZ}(X)$ and $E\geq 0$ by homogeneity (see \cite[Corollary~3.25]{IkomaCont} and Corollary~\ref{cor:Yuan_main} \eqref{item: cor Yuan main homogeneity}).
Let $\overline{A}$ be any adelic Cartier divisor on $X$ such that $\widehat{\Gamma}^{\rm s}_{X|Y}\left(\overline{A}\right)\neq\{0\}$ and $\widehat{\Gamma}^{\rm s}_{X|Y}(\overline{A};Y)\neq\{0\}$, and set $\overline{A}'\coloneqq \overline{A}(\log(2))$.
By continuity of the arithmetic volume functions (see \cite[Main Theorem]{IkomaCont} and Remark~\ref{rem:arpin} \eqref{item: continuity arpin}), one has
\[
\frac{\avol\left(\overline{D};E\right)-\avol\left(\overline{D};E+rY\right)}{r}\leq (\dim X+1)\left.\left\langle\left(\overline{D}+r\overline{A}';E+rY\right)^{\cdot\dim X}\right\rangle\right|_Y
\]
for any sufficiently small $r\in\RR_{>0}$ by Theorem~\ref{thm:FE} and Proposition~\ref{prop: arithmetic Fujita approximation}.
By taking $r\downarrow 0$, one obtains the assertion by using Remark~\ref{rem:arpin} \eqref{item: continuity arpin} again.
\end{proof}

\subsection{Estimation of Newton--Okounkov bodies}\label{subsec: estimation of NO bodies}

\begin{definition}\label{defn: Estimates II prelim}
Let $\mathscr{X}$ be a normal and projective $O_K$-model of $X$ such that there exists an arithmetic Cartier divisor $\overline{\mathscr{A}}$ on $\mathscr{X}$ such that $\overline{A}\leq\overline{\mathscr{A}}^{\rm ad}$ and $\mathscr{A}\cap X=A$ and such that the Zariski closure $\mathscr{Y}$ of $Y$ in $\mathscr{X}$ is Cartier.

Let $\rho_0\in\QQ_{>0}$ be as in Proposition~\ref{prop: Estimates II inclusions}.
By Lemma~\ref{lem:construction_of_the_metric}, there exists a $\mathscr{A}$-Green function $g^{\overline{\mathscr{A}}'}$ and a $\mathscr{Y}$-Green function $g^{\overline{\mathscr{Y}}}$ such that
\[
\widehat{\Gamma}^{\rm ss}_{X|Y}\left(m\overline{A};nY\right)\subset\widehat{\Gamma}^{\rm s}\left(\left.\mathcal{O}_{\mathscr{X}}\left(m\overline{\mathscr{A}}'-n\overline{\mathscr{Y}}\right)\right|_{\mathscr{Y}}\right)
\]
for every $m,n\in\ZZ_{\geq 0}$, where $\overline{\mathscr{A}}'\coloneqq\left(\mathscr{A},g^{\overline{\mathscr{A}}'}\right)$.
Let $\overline{\mathscr{M}}_0$ be a $\mathscr{Y}$-effective arithmetic Cartier divisor on $\mathscr{X}$ such that $\mathcal{O}_{\mathscr{X}}\left(\overline{\mathscr{M}}_0+\overline{\mathscr{Y}}\right)$ is also $\mathscr{Y}$-effective, and set $\overline{\mathscr{M}}\coloneqq \overline{\mathscr{A}}+\lceil\rho_0\rceil\overline{\mathscr{M}}_0$.
Then we obtain a natural inclusion
\[
\widehat{\Gamma}_{X|Y}^{\rm ss}\left(m\overline{A};mrY\right)\subset\widehat{\Gamma}^{\rm s}\left(\left.\mathcal{O}_{\mathscr{X}}\left(m\overline{\mathscr{M}}\right)\right|_{\mathscr{Y}}\right)
\]
for any $m\in\ZZ_{\geq 0}$ and $r\in\RR$ with $0\leq r\leq\rho_0$.

Fix a flag
\[
\mathscr{F}_{\sbullet}\colon \mathscr{F}_0\coloneqq \mathscr{X}\supset\mathscr{F}_1\coloneqq \mathscr{Y}\supset\mathscr{F}_2\supset\dots\supset\mathscr{F}_{\dim\mathscr{X}}=\{\xi\}
\]
on $\mathscr{X}$ such that
\[
\mathscr{F}_{\geq 1}\colon \mathscr{Y}=\mathscr{F}_1\supset\mathscr{F}_2\supset\dots\supset\mathscr{F}_{\dim\mathscr{X}}=\{\xi\}
\]
is a good flag on $\mathscr{Y}$ over a prime number $p$ (see Definition~\ref{defn: good flag}), and let $\bm{w}_{\mathscr{F}_{\sbullet}}$ denote the valuation map attached to $\mathscr{F}_{\sbullet}$.

Given an adelic $\RR$-Cartier divisor $\overline{M}$ on $X$, an $\RR$-Cartier divisor $N$ on $X$, a prime number $p$, and an $\varepsilon\in\RR_{>0}$, we set
\begin{equation}
\widetilde{C}\left(\overline{M},N,X,p,\varepsilon\right)\coloneqq \frac{[K:\QQ]I(M)\left(\log(4)\delta\left(\overline{M}\right)+\varepsilon\right)}{\log(p)}+\varepsilon [K:\QQ]I(N)
\end{equation}
(see Definition~\ref{defn:BasicConstants}).
\end{definition}

\begin{theorem}\label{thm:Estimates II}
Let $X$ be a normal, projective, and geometrically connected $K$-variety, let $Y$ be a prime Cartier divisor on $X$, and let $\overline{A}$ be a w-ample adelic Cartier divisor on $X$.
We use the same notation as in Definition~\ref{defn: Estimates II prelim}.
Let $\rho_0$ be as in Proposition~\ref{prop: Estimates II inclusions}.
Let $\overline{\mathscr{M}}$ be a $\mathscr{Y}$-effective arithmetic Cartier divisor on $\mathscr{X}$ such that $H^0_{X|Y}(\mathscr{M}|_X+Y)\neq\{0\}$ and such that
\[
\widehat{\Gamma}_{X|Y}^{\rm f}\left(m\overline{A};mrY\right)\subset H^0\left(\left.\mathcal{O}_{\mathscr{X}}\left(m\mathscr{M}\right)\right|_{\mathscr{Y}}\right)
\]
and
\[
\widehat{\Gamma}_{X|Y}^{\rm ss}\left(m\overline{A};mrY\right)\subset\widehat{\Gamma}^{\rm s}\left(\left.\mathcal{O}_{\mathscr{X}}\left(m\overline{\mathscr{M}}\right)\right|_{\mathscr{Y}}\right)
\]
for any $m\in\ZZ_{\geq 0}$ and $r\in\RR$ with $0\leq r\leq\rho_0$.
Fix an $\varepsilon\in\RR$ with $0<\varepsilon\leq 1$.
There then exist a $\lambda(\varepsilon,p)\in\ZZ_{\geq 1}$, which depends on $\overline{A}$, $\mathscr{M}$, $Y$, $X$, $\varepsilon$, and $p$, and positive real numbers $S$, $S'$, which depend only on $A$, $\overline{\mathscr{M}}$, $Y$, and $X$, such that
\begin{align*}
&-r\widetilde{C}\left(\left.\overline{\mathscr{M}}^{\rm ad}\right|_Y,A|_Y+\rho_0\mathscr{M}|_Y,Y,p,\varepsilon\right)m^{\dim X+1} \\
&\qquad\qquad\qquad\qquad\qquad\qquad -[K:\QQ]I(A)m^{\dim X}\log(m)-Sm^{\dim X} \\
&\qquad\qquad \leq\ahss\left(m\overline{A}\right)-\ahss\left(m\overline{A};mrY\right)-\#\bm{w}_{\mathscr{F}_{\sbullet}}\left(\widehat{\Gamma}^{\rm ss}\left(m\overline{A}\right)\setminus\{0\}\right)\log(p) \\
&\qquad\qquad\qquad\qquad\qquad\qquad +\#\bm{w}_{\mathscr{F}_{\sbullet}}\left(\widehat{\Gamma}^{\rm ss}\left(m\overline{A};mrY\right)\setminus\{0\}\right)\log(p) \\
&\qquad\qquad \leq r\widetilde{C}\left(\left.\overline{\mathscr{M}}^{\rm ad}\right|_Y,A|_Y+\rho_0\mathscr{M}|_Y,Y,p,\varepsilon\right)m^{\dim X+1}+S'm^{\dim X}
\end{align*}
for any $m\in\ZZ$ and $r\in\RR$ with $m\geq \lambda(\varepsilon,p)$ and $0<r\leq\rho_0$, respectively.
\end{theorem}

\begin{proof}
Let $\overline{M}\coloneqq \overline{\mathscr{M}}^{\rm ad}$.
We divide the proof into three steps.

\Proofstep
Obviously, we have
\begin{align*}
&\#\bm{w}_{\mathscr{F}_{\sbullet}}\left(\widehat{\Gamma}^{\rm ss}\left(m\overline{A}\right)\setminus\{0\}\right)-\#\bm{w}_{\mathscr{F}_{\sbullet}}\left(\widehat{\Gamma}^{\rm ss}\left(m\overline{A};mrY\right)\setminus\{0\}\right) \\
&\qquad\qquad\qquad\qquad\qquad\qquad =\sum_{n=0}^{\lceil mr\rceil-1}\#\bm{w}_{\mathscr{F}_{\geq 1}}\left(\widehat{\Gamma}_{X|Y}^{\rm ss}\left(m\overline{A};nY\right)\setminus\{0\}\right).
\end{align*}
For the given $\varepsilon\in\RR_{>0}$, we can find a $\lambda'(\varepsilon,p)\in\ZZ_{\geq 1}$ such that
\begin{equation}
\log(4p)\left(\log(4p)+\log\left(I(M|_Y)m^{\dim Y}\right)\right)\leq\varepsilon m
\end{equation}
for any $m\in\ZZ$ with $m\geq\lambda'(\varepsilon,p)$.
By Proposition~\ref{prop: Estimates II inclusions}, there exists an integer $\lambda(\varepsilon,p)$ such that $\lambda(\varepsilon,p)\geq\max\left\{\lambda'(\varepsilon,p),1/\rho_0\right\}$ and such that
\[
\widehat{\Gamma}_{\quot(X|Y)}^{\rm ss}\left(m\left(\overline{A}(-\varepsilon)\right);nY\right)\subset\widehat{\Gamma}_{X|Y}^{\rm ss}\left(m\overline{A};nY\right)
\]
and
\[
\widehat{\Gamma}_{\quot(X|Y)}^{\rm ss}\left(m\overline{A};nY\right)\subset\widehat{\Gamma}_{X|Y}^{\rm ss}\left(m\left(\overline{A}(\varepsilon)\right);nY\right)
\]
hold for any $m,n\in\ZZ_{\geq 1}$ with $n/m\leq\rho_0$ and $m\geq\lambda(\varepsilon,p)$.

Put
\begin{equation}
\widetilde{C}_0\left(\overline{M}|_Y,Y,\varepsilon\right)\coloneqq [K_Y:\QQ]I(M|_Y)\left(\log(4)\delta\left(\overline{M}|_Y\right)+\varepsilon\right).
\end{equation}
By Theorem~\ref{thm:YuanMoriwaki} and Remark~\ref{rem:Yuan_rescale}, we have
\begin{align}
&\sum_{n=0}^{\lceil mr\rceil-1}\ahss_{\quot(X|Y)}\left(m\overline{A};nY\right)-(\varepsilon m+\log(3))\sum_{n=0}^{\lceil mr\rceil-1}\rk_{\QQ} H^0_{X|Y}(mA-nY) \nonumber\\
&\qquad\qquad\qquad\qquad\qquad\qquad -\frac{\widetilde{C}_0\left(\overline{M}|_Y,Y,\varepsilon\right)}{\log(p)}m^{\dim X}(mr+1) \nonumber\\
&\qquad \leq\sum_{n=0}^{\lceil mr\rceil-1}\ahss_{\quot(X|Y)}\left(m\left(\overline{A}(-\varepsilon)\right);nY\right)-\frac{\widetilde{C}_0\left(\overline{M}|_Y,Y,\varepsilon\right)}{\log(p)}m^{\dim X}(mr+1) \nonumber\\
&\qquad \leq\sum_{n=0}^{\lceil mr\rceil-1}\#\bm{w}_{\mathscr{F}_{\geq 1}}\left(\widehat{\Gamma}_{\quot(X|Y)}^{\rm ss}\left(m\left(\overline{A}(-\varepsilon)\right);nY\right)\setminus\{0\}\right)\log(p) \nonumber\\
&\qquad \leq\#\bm{w}_{\mathscr{F}_{\sbullet}}\left(\widehat{\Gamma}^{\rm ss}\left(m\overline{A}\right)\setminus\{0\}\right)\log(p) \nonumber\\
&\qquad\qquad\qquad\qquad\qquad\qquad -\#\bm{w}_{\mathscr{F}_{\sbullet}}\left(\widehat{\Gamma}^{\rm ss}\left(m\overline{A};mrY\right)\setminus\{0\}\right)\log(p) \nonumber\\
&\qquad \leq\sum_{n=0}^{\lceil mr\rceil-1}\#\bm{w}_{\mathscr{F}_{\geq 1}}\left(\widehat{\Gamma}_{\quot(X|Y)}^{\rm ss}\left(m\overline{A};nY\right)\setminus\{0\}\right)\log(p) \nonumber\\
&\qquad \leq\sum_{n=0}^{\lceil mr\rceil-1}\ahss_{\quot(X|Y)}\left(m\overline{A};nY\right)+\frac{\widetilde{C}_0\left(\overline{M}|_Y,Y,\varepsilon\right)}{\log(p)}m^{\dim X}(mr+1). \label{eqn: Estimates II Step 1}
\end{align}

\Proofstep
In this step, we show the lower bound.
For that purpose, we consider the filtration by natural inclusions
\[
H^0(mA)\supset H^0(mA-Y)\supset\dots\supset H^0(mA-\lceil mr\rceil Y)\supset\{0\}
\]
and apply Lemma~\ref{lem:filtration_est} to it.
By using Remark~\ref{rem:Yuan_rescale} again, we have
\begin{align*}
&\ahss\left(m\overline{A}\right)-\ahss\left(m\overline{A};mrY\right) \\
&\qquad \geq\ahss\left(\left(m\overline{A}\right)(\log(\lceil mr\rceil))\right)-\ahss\left(m\overline{A};mrY\right)-\log(3\lceil mr\rceil)\rk_{\QQ} H^0(mA) \\
&\qquad \geq\sum_{n=0}^{\lceil mr\rceil-1}\ahss_{X|Y}\left(m\overline{A};nY\right)-\log(3\lceil mr\rceil)\rk_{\QQ} H^0(mA) \\
&\qquad \geq\sum_{n=0}^{\lceil mr\rceil-1}\ahss_{\quot(X|Y)}\left(m\left(\overline{A}(-\varepsilon)\right);nY\right)-\log(3\lceil mr\rceil)\rk_{\QQ} H^0(mA) \\
&\qquad \geq\sum_{n=0}^{\lceil mr\rceil-1}\ahss_{\quot(X|Y)}\left(m\overline{A};nY\right)-(\varepsilon m+\log(3))\sum_{n=0}^{\lceil mr\rceil-1}\rk_{\QQ} H^0_{X|Y}(mA-nY) \\
&\qquad\qquad\qquad\qquad\qquad\qquad -\log(3\lceil mr\rceil)\rk_{\QQ} H^0(mA),
\end{align*}
which implies by \eqref{eqn: Estimates II Step 1} that
\begin{align*}
&\ahss\left(m\overline{A}\right)-\ahss\left(m\overline{A};mrY\right)-\#\bm{w}_{\mathscr{F}_{\sbullet}}\left(\widehat{\Gamma}^{\rm ss}\left(m\overline{A}\right)\setminus\{0\}\right)\log(p) \\
&\qquad\qquad\qquad\qquad\qquad\qquad +\#\bm{w}_{\mathscr{F}_{\sbullet}}\left(\widehat{\Gamma}^{\rm ss}\left(m\overline{A};mrY\right)\setminus\{0\}\right)\log(p) \\
&\quad \geq -\frac{\widetilde{C}_0\left(\overline{M}|_Y,Y,\varepsilon\right)}{\log(p)}m^{\dim X}(mr+1) \\
&\qquad -(\varepsilon m+\log(3))\sum_{n=0}^{\lceil mr\rceil-1}\rk_{\QQ} H^0_{X|Y}(mA-nY)-\log(3\lceil mr\rceil)\rk_{\QQ} H^0(mA) \\
&\quad \geq -r\widetilde{C}\left(\overline{M}|_Y,A|_Y+\rho_0M|_Y,Y,p,\varepsilon\right)m^{\dim X+1}-[K:\QQ]I(A)m^{\dim X}\log(m) \\
&\qquad\qquad\qquad\qquad\qquad\qquad -Sm^{\dim X}
\end{align*}
for any $m\geq\lambda(\varepsilon,p)$.
Here we set
\begin{align*}
&S\coloneqq \frac{\widetilde{C}_0\left(\overline{M}|_Y,Y,1\right)}{\log(2)}+(1+(\rho_0+1)\log(3))[K_Y:\QQ]I(A|_Y+\rho_0M|_Y) \\
&\qquad\qquad\qquad\qquad\qquad\qquad\qquad +\log(3(\rho_0+1))[K:\QQ]I(A).
\end{align*}

\Proofstep
Applying Remark~\ref{rem:exact_sequence} \eqref{item: exact sequences combined} to the exact sequence
\[
0\to H^0(mA-\lceil mr\rceil Y)\xrightarrow{\otimes 1_Y^{\otimes \lceil mr\rceil}} H^0(mA)\to H^0_{X|\lceil mr\rceil Y}(mA)\to 0,
\]
we obtain
\begin{align}
&-\log(6)\rk_{\QQ} H^0(mA)\leq\ahss\left(m\overline{A}\right)-\ahss\left(m\overline{A};mrY\right)-\ahss_{X|\lceil mr\rceil Y}\left(m\overline{A}\right) \nonumber\\
&\qquad\qquad\qquad\qquad\qquad\qquad \leq\log(6)\rk_{\QQ} H^0(mA-\lceil mr\rceil Y).
\end{align}
By applying Lemma~\ref{lem:quot_exact} to the diagram
\[
\xymatrix{0 \ar[r] & H^0_{X|Y}(mA-nY) \ar[r] & H^0_{X|(n+1)Y}(mA) \ar[r] & H^0_{X|nY}(mA) \ar[r] & 0 \\
0 \ar[r] & H^0(mA-nY) \ar[r]^-{\otimes 1_Y^{\otimes n}} \ar[u] & H^0(mA) \ar[r] \ar[u] & H_{X|nY}^0(mA) \ar[r] \ar@{=}[u] & 0
}
\]
for each $n\in\ZZ_{\geq 1}$, we have
\begin{align}
&\ahss_{X|(n+1)Y}\left(m\overline{A}\right)-\ahss_{X|nY}\left(m\overline{A}\right)\leq\ahss_{X|Y}\left(\left(m\overline{A}\right)(\log(2));nY\right) \nonumber\\
&\qquad\qquad \leq\ahss_{\quot(X|Y)}\left(m\overline{A};nY\right)+\log(6)\rk_{\QQ} H^0_{X|Y}(mA-nY) \label{eqn:Estimates II 1}
\end{align}
by Remark~\ref{rem:Yuan_rescale}.
Therefore, by summing up \eqref{eqn:Estimates II 1} for $n=1,\dots,\lceil mr\rceil-1$, we have
\begin{align*}
&\ahss_{X|\lceil mr\rceil Y}\left(m\overline{A}\right) \\
&\qquad \leq\sum_{n=0}^{\lceil mr\rceil-1}\ahss_{\quot(X|Y)}\left(m\overline{A};nY\right)+\log(6)\sum_{n=1}^{\lceil mr\rceil-1}\rk_{\QQ} H^0_{X|Y}(mA-nY),
\end{align*}
which leads to the upper bound of the theorem as
\begin{align*}
&\ahss\left(m\overline{A}\right)-\ahss\left(m\overline{A};mrY\right)-\#\bm{w}_{\mathscr{F}_{\sbullet}}\left(\widehat{\Gamma}^{\rm ss}\left(m\overline{A}\right)\setminus\{0\}\right)\log(p) \\
&\qquad\qquad\qquad\qquad\qquad\qquad +\#\bm{w}_{\mathscr{F}_{\sbullet}}\left(\widehat{\Gamma}^{\rm ss}\left(m\overline{A};mrY\right)\setminus\{0\}\right)\log(p) \\
& \leq \frac{\widetilde{C}_0\left(\overline{M}|_Y,Y,\varepsilon\right)}{\log(p)}m^{\dim X}(mr+1)+(\varepsilon m+\log(18))\sum_{n=0}^{\lceil mr\rceil-1}\rk_{\QQ} H^0_{X|Y}(mA-nY) \\
&\qquad\qquad\qquad\qquad\qquad\qquad +\log(6)\rk_{\QQ} H^0(mA-\lceil mr\rceil Y) \\
& \leq r\widetilde{C}\left(\overline{M}|_Y,A|_Y+\rho_0M|_Y,Y,p,\varepsilon\right)m^{\dim X+1}+S'm^{\dim X}
\end{align*}
for any $m\geq\lambda(\varepsilon,p)$ by \eqref{eqn: Estimates II Step 1}.
Here we set
\begin{align*}
&S'\coloneqq \frac{\widetilde{C}_0\left(\overline{M}|_Y,Y,1\right)}{\log(2)}+(1+(\rho_0+1)\log(18))[K_Y:\QQ]I(A|_Y+\rho_0M|_Y) \\
&\qquad\qquad\qquad\qquad\qquad\qquad\qquad +\log(6)[K:\QQ]I(A).
\end{align*}
\end{proof}

\subsection{Yuan-type inequality}\label{subsec:Yuan}

\begin{definition}\label{defn:Yuan_Moriwaki_Newton_Okounkov}
Let $X$ be a normal, projective, and geometrically connected $K$-variety, let $Y$ be a prime Cartier divisor on $X$, and let $\left(\overline{D};E\right)\in\aDDiv_{\ZZ,\ZZ}(X)$ be a $Y$-big pair on $X$.
We choose a $\rho_0'\in\QQ_{>0}$ such that $\left(\overline{D};E+rY\right)$ is $Y$-big for every $r\in\RR$ with $0\leq r\leq \rho_0'$.
As in Definition~\ref{defn: Estimates II prelim}, there exists a normal and projective $O_K$-model $\mathscr{X}$ of $X$ such that the Zariski closure $\mathscr{Y}$ of $Y$ in $\mathscr{X}$ is Cartier and such that there exists a $\mathscr{Y}$-effective arithmetic Cartier divisor $\overline{\mathscr{M}}$ on $\mathscr{X}$ such that
\[
\widehat{\Gamma}_{X|Y}^{\rm f}\left(m\overline{D};mE+mrY\right)\subset H^0\left(\left.\mathcal{O}_{\mathscr{X}}\left(m\mathscr{M}\right)\right|_{\mathscr{Y}}\right)
\]
and
\[
\widehat{\Gamma}_{X|Y}^{\rm ss}\left(m\overline{D};mE+mrY\right)\subset \widehat{\Gamma}^{\rm s}\left(\left.\mathcal{O}_{\mathscr{X}}\left(m\overline{\mathscr{M}}\right)\right|_{\mathscr{Y}}\right)
\]
for every $m\in\ZZ_{\geq 0}$ and $r\in\RR$ with $0\leq r\leq \rho_0'$.
Let $\mathscr{F}_{\sbullet}$ be as in Definition~\ref{defn: Estimates II prelim}, and let $\bm{w}_{\mathscr{F}_{\sbullet}}$ denote the valuation map attached to $\mathscr{F}_{\sbullet}$.
We then define
\[
\widehat{\Delta}^{\mathscr{F}_{\sbullet}}_{\rm YM}\left(\overline{D};E\right)\coloneqq \overline{\left(\bigcup_{m\in\ZZ_{\geq 1}}\frac{1}{m}\bm{w}_{\mathscr{F}_{\sbullet}}\left(\widehat{\Gamma}^{\rm ss}\left(m\overline{D};mE\right)\setminus\{0\}\right)\right)}.
\]
More generally, for any $Y$-big pair $\left(\overline{D};E\right)\in\aDDiv_{\QQ,\QQ}(X)$, we define
\[
\widehat{\Delta}^{\mathscr{F}_{\sbullet}}_{\rm YM}\left(\overline{D};E\right)\coloneqq \frac{1}{n}\widehat{\Delta}^{\mathscr{F}_{\sbullet}}_{\rm YM}\left(n\overline{D};nE\right),
\]
where $n$ denotes any positive integer such that $\left(n\overline{D};nE\right)\in\aDDiv_{\ZZ,\ZZ}(X)$.
\end{definition}

\begin{lemma}\label{lem:FE3_bounded}
We use the same notation as in Definition~\ref{defn:Yuan_Moriwaki_Newton_Okounkov}.
\begin{enumerate}
\item $\widehat{\Delta}^{\mathscr{F}_{\sbullet}}_{\rm YM}\left(\overline{D};E\right)$ is a compact convex body in $\RR^{\dim X+1}$.
\item If $\left(\overline{D};E\right)\in\aDDiv_{\ZZ,\ZZ}(X)$, then one has
\[
\vol_{\RR^{\dim X+1}}\left(\widehat{\Delta}^{\mathscr{F}_{\sbullet}}_{\rm YM}\left(\overline{D};E\right)\right)=\lim_{\substack{m\in\ZZ_{\geq 1}, \\ m\to\infty}}\frac{\#\bm{w}_{\mathscr{F}_{\sbullet}}\left(\widehat{\Gamma}^{\rm ss}\left(m\overline{D};mE\right)\setminus\{0\}\right)}{m^{\dim X+1}}\in\RR_{>0}.
\]
\end{enumerate}
\end{lemma}

\begin{proof}
(1): We may assume that $\left(\overline{D};E\right)\in\aDDiv_{\ZZ,\ZZ}(X)$.
Let $B$ be an ample Cartier divisor on $X$.
If $\widehat{\Gamma}^{\rm ss}\left(m\overline{D};nE\right)\neq\{0\}$, then
\[
\deg\left((mD-mE-nY)\cdot B^{\cdot(\dim X-1)}\right)\geq 0.
\]
Thus
\[
-m\ord_Y(D-E)\leq w_1(\phi)\leq m\frac{\deg\left((D-E)\cdot B^{\cdot(\dim X-1)}\right)}{\deg\left(Y\cdot B^{\cdot(\dim X-1)}\right)}
\]
for any $\phi\in\widehat{\Gamma}^{\rm ss}\left(m\overline{D};mE\right)\setminus\{0\}$ (see Notation and terminology~\ref{NC:flag}).
Moreover, by \cite[Lemma~2.4]{Yuan09} applied to $\left.\mathcal{O}_{\mathscr{X}}\left(\overline{\mathscr{M}}\right)\right|_{\mathscr{Y}}$, one can find $a,b\in\RR_{>0}$ such that $\widehat{\Delta}^{\mathscr{F}_{\sbullet}}_{\rm YM}\left(\overline{D};E\right)\subset [-a,b]^{\dim X+1}$.

(2): Since $\left(\overline{D};E\right)$ is $Y$-big, the semigroup
\[
\left\{\left(\bm{w}_{\mathscr{F}_{\sbullet}}(\phi),m\right)\,\colon \,\phi\in\widehat{\Gamma}^{\rm ss}\left(m\overline{D};mE\right)\setminus\{0\},\,m\in\ZZ_{\geq 0}\right\}
\]
generates $\ZZ^{\dim X+2}$ (see \cite[Proposition~5.2]{MoriwakiEst}).
Hence the assertions follow from \cite[Proposition~2.1]{Lazarsfeld_Mustata08} (see also \cite[Corollaire~1.14]{BoucksomBourbaki}).
\end{proof}

\begin{theorem}\label{thm: general lower bound}
Let $X$ be a normal, projective, and geometrically connected $K$-variety, let $Y$ be a prime Cartier divisor on $X$, let $\left(\overline{D};E\right)\in\aDDiv_{\ZZ,\ZZ}(X)$ be a $Y$-big pair on $X$, and let $\overline{A}\in\aDiv_{\ZZ}(X)$ be a w-ample adelic Cartier divisor on $X$.
Suppose that $\left(\overline{D}-\overline{A};E\right)$ is $Y$-big.
Let $\rho_0$ be as in Proposition~\ref{prop: Estimates II inclusions}.
Let $\overline{\mathscr{M}}$ be a $\mathscr{Y}$-effective arithmetic Cartier divisor on $\mathscr{X}$ such that $H^0_{X|Y}(\mathscr{M}|_X+Y)\neq\{0\}$ and such that
\[
\widehat{\Gamma}_{X|Y}^{\rm f}\left(m\overline{A};mrY\right)\subset H^0\left(\left.\mathcal{O}_{\mathscr{X}}\left(m\mathscr{M}\right)\right|_{\mathscr{Y}}\right)
\]
and
\[
\widehat{\Gamma}_{X|Y}^{\rm ss}\left(m\overline{A};mrY\right)\subset\widehat{\Gamma}^{\rm s}\left(\left.\mathcal{O}_{\mathscr{X}}\left(m\overline{\mathscr{M}}\right)\right|_{\mathscr{Y}}\right)
\]
for every $m\in\ZZ_{\geq 0}$ and $r\in\RR$ with $0\leq r\leq\rho_0$.
Fix an $\varepsilon\in\RR$ with $0<\varepsilon\leq 1$.
There then exist a $\mu(\varepsilon,p)\in\ZZ_{\geq 1}$, which depends on $\left(\overline{D};E\right)$, $\overline{A}$, $\mathscr{M}$, $Y$, $X$, $\varepsilon$, and $p$, and a positive real number $S$, which depends only on $D$, $A$, $\overline{\mathscr{M}}$, $Y$, and $X$, such that
\begin{align*}
&\ahss\left(m\overline{D};mE\right)-\ahss\left(m\overline{D};mE+mrY\right)-\#\bm{w}_{\mathscr{F}_{\sbullet}}\left(\widehat{\Gamma}^{\rm ss}\left(m\overline{A}\right)\setminus\{0\}\right)\log(p) \\
&\qquad\qquad\qquad\qquad\qquad\qquad +\#\bm{w}_{\mathscr{F}_{\sbullet}}\left(\widehat{\Gamma}^{\rm ss}\left(m\overline{A};mrY\right)\setminus\{0\}\right)\log(p) \\
&\qquad\qquad \geq -r\widetilde{C}\left(\left.\overline{\mathscr{M}}^{\rm ad}\right|_Y,A|_Y+\rho_0\mathscr{M}|_Y,Y,p,\varepsilon\right)m^{\dim X+1} \\
&\qquad\qquad\qquad\qquad\qquad\qquad -[K:\QQ]I(D)m^{\dim X}\log(m)-Sm^{\dim X}
\end{align*}
for every $m\in\ZZ$ and $r\in\RR$ with $m\geq \mu(\varepsilon,p)$ and $0<r\leq\rho_0$, respectively.
\end{theorem}

\begin{proof}
Let $\overline{M}\coloneqq \overline{\mathscr{M}}^{\rm ad}$.
Given any $\varepsilon\in\RR_{>0}$, there exists a $\mu'(\varepsilon)\in\ZZ_{\geq 1}$, which depends on $\left(\overline{D};E\right)$, $\overline{A}$, $Y$, $X$, and $\varepsilon$, such that
\begin{align}
&\widehat{\Gamma}^{\rm ss}_{\quot(X|Y)}\left(m\overline{A};nY\right)\subset\widehat{\Gamma}^{\rm ss}_{X|Y}\left(m(\overline{A}(\varepsilon));nY\right), \\
&\widehat{\Gamma}^{\rm ss}_{\quot(X|Y)}\left(m(\overline{A}(-\varepsilon));nY\right)\subset\widehat{\Gamma}^{\rm ss}_{X|Y}\left(m\overline{A};nY\right),
\end{align}
and
\begin{equation}
\widehat{\Gamma}^{\rm s}_{X|Y}\left(m\overline{D}-m\overline{A};mE\right)\neq\{0\}
\end{equation}
for every $m,n\in\ZZ_{\geq 1}$ with $n/m\leq\rho_0$ and $m\geq\mu'(\varepsilon)$.
Let $\lambda(\varepsilon,p)$ be as in the proof of Theorem~\ref{thm:Estimates II}, and set
\[
\mu(\varepsilon,p)\coloneqq \max\{\lambda(\varepsilon,p),\mu'(\varepsilon)\}.
\]

By the same arguments as in the proof of Theorem~\ref{thm:Estimates II}, we obtain
\begin{align*}
&\ahss\left(m\overline{D};mE\right)-\ahss\left(m\overline{D};mE+mrY\right) \\
&\qquad \geq\ahss\left((m\overline{D})(\log(\lceil mr\rceil));mE\right)-\ahss\left(m\overline{D};mE+mrY\right) \\
&\qquad\qquad\qquad -\log(3\lceil mr\rceil)\rk_{\QQ} H^0(mD) \\
&\qquad \geq\sum_{n=0}^{\lceil mr\rceil-1}\ahss_{X|Y}\left(m\overline{D};mE+nY\right)-\log(3\lceil mr\rceil)\rk_{\QQ} H^0(mD) \\
&\qquad \geq\sum_{n=0}^{\lceil mr\rceil-1}\ahss_{X|Y}\left(m\overline{A};nY\right)-\log(3\lceil mr\rceil)\rk_{\QQ}H^0(mD) \\
&\qquad \geq\sum_{n=0}^{\lceil mr\rceil-1}\ahss_{\quot(X|Y)}\left(m\overline{A};nY\right)-(\varepsilon m+\log(3)) \\
&\qquad\qquad\qquad \times\sum_{n=0}^{\lceil mr\rceil-1}\rk_{\QQ} H^0_{X|Y}(mA-nY)-\log(3\lceil mr\rceil)\rk_{\QQ} H^0(mD)
\end{align*}
and
\begin{align*}
&\#\bm{w}_{\mathscr{F}_{\sbullet}}\left(\widehat{\Gamma}^{\rm ss}\left(m\overline{A}\right)\setminus\{0\}\right)\log(p)-\#\bm{w}_{\mathscr{F}_{\sbullet}}\left(\widehat{\Gamma}^{\rm ss}\left(m\overline{A};mrY\right)\setminus\{0\}\right)\log(p) \\
&\qquad \leq\sum_{n=0}^{\lceil mr\rceil-1}\#\bm{w}_{\mathscr{F}_{\geq 1}}\left(\widehat{\Gamma}_{\quot(X|Y)}^{\rm ss}\left(m\overline{A};nY\right)\setminus\{0\}\right)\log(p) \\
&\qquad \leq\sum_{n=0}^{\lceil mr\rceil-1}\ahss_{\quot(X|Y)}\left(m\overline{A};nY\right)+\frac{\widetilde{C}_0\left(\overline{M}|_Y,Y,\varepsilon\right)}{\log(p)}m^{\dim X}(mr+1).
\end{align*}
Hence,
\begin{align*}
&\ahss\left(m\overline{D};mE\right)-\ahss\left(m\overline{D};mE+mrY\right)-\#\bm{w}_{\mathscr{F}_{\sbullet}}\left(\widehat{\Gamma}^{\rm ss}\left(m\overline{A}\right)\setminus\{0\}\right)\log(p) \\
&\quad\qquad\qquad\qquad\qquad +\#\bm{w}_{\mathscr{F}_{\sbullet}}\left(\widehat{\Gamma}^{\rm ss}\left(m\overline{A};mrY\right)\setminus\{0\}\right)\log(p) \\
&\quad \geq -r\widetilde{C}\left(\overline{M}|_Y,A|_Y+\rho_0M|_Y,Y,p,\varepsilon\right)m^{\dim X+1}-[K:\QQ]I(D)m^{\dim X}\log(m) \\
&\quad\qquad\qquad\qquad\qquad -Sm^{\dim X}
\end{align*}
for any $m\geq\mu(\varepsilon,p)$, where we set
\begin{align*}
&S\coloneqq \frac{\widetilde{C}_0\left(\overline{M}|_Y,Y,1\right)}{\log(2)}+(1+(\rho_0+1)\log(3))[K_Y:\QQ]I(A|_Y+\rho_0M|_Y) \\
&\qquad\qquad\qquad\qquad\qquad\qquad\qquad +\log(3(\rho_0+1))[K:\QQ]I(D).
\end{align*}
\end{proof}

\begin{theorem}\label{thm:FE3}
Let $X$ be a normal, projective, and geometrically connected $K$-variety, let $Y$ be a prime Cartier divisor on $X$, and let $\left(\overline{D};E\right)\in\aDDiv_{\QQ,\QQ}(X)$ be a $Y$-big pair on $X$.
Let $\overline{A}\in\aDiv_{\QQ}(X)$ be any w-ample adelic $\QQ$-Cartier divisor on $X$ such that $\left(\overline{D}-\overline{A};E\right)$ is $Y$-big.
If $\ord_Y(E)>0$, then
\[
\lim_{r\downarrow 0}\frac{\avol\left(\overline{D};E\right)-\avol\left(\overline{D};E+rY\right)}{r}\geq (\dim X+1)\avolq{X|Y}\left(\overline{A}\right).
\]
\end{theorem}

\begin{proof}
By homogeneity (see \cite[Corollary~3.25]{IkomaCon} and Corollary~\ref{cor:Yuan_main} \eqref{item: cor Yuan main homogeneity}), one can assume that $\left(\overline{D};E\right)\in\aDDiv_{\ZZ,\ZZ}(X)$ and $\overline{A}\in\aDiv_{\ZZ}(X)$.
For the $\overline{A}$, one can choose a $\rho_0$ as in Proposition~\ref{prop: Estimates II inclusions}.
There exists a normal and projective $O_K$-model $\mathscr{X}$ such that the Zariski closure $\mathscr{Y}$ of $Y$ in $\mathscr{X}$ is Cartier and such that there exists a $\mathscr{Y}$-effective arithmetic Cartier divisor $\overline{\mathscr{M}}$ on $\mathscr{X}$ such that $H^0_{X|Y}(\mathscr{M}|_X+Y)\neq\{0\}$ and such that
\[
\widehat{\Gamma}^{\rm f}_{X|Y}\left(m\overline{A};mrY\right)\subset H^0\left(\left.\mathcal{O}_{\mathscr{X}}\left(m\mathscr{M}\right)\right|_{\mathscr{Y}}\right)
\]
and
\[
\widehat{\Gamma}^{\rm ss}_{X|Y}\left(m\overline{A};mrY\right)\subset\widehat{\Gamma}^{\rm s}\left(\left.\mathcal{O}_{\mathscr{X}}\left(m\overline{\mathscr{M}}\right)\right|_{\mathscr{Y}}\right)
\]
hold for every $m\in\ZZ_{\geq 0}$ and $r\in\RR$ with $0\leq r\leq\rho_0$.
Let $\overline{M}\coloneqq \overline{\mathscr{M}}^{\rm ad}$.

Fix any $\varepsilon\in\RR$ with $0<\varepsilon\leq 1$ such that $\overline{A}(-\varepsilon)$ is also w-ample.
Let $p$ be any prime number satisfying
\[
p\geq\max\left\{3,4^{\delta\left(\overline{M}|_Y\right)/\varepsilon}\right\},
\]
and let
\[
\mathscr{F}_{\sbullet}\colon \mathscr{X}\supset\mathscr{Y}\supset\mathscr{F}_2\supset\dots\supset\mathscr{F}_{\dim X+1}=\{\xi\}
\]
be a flag on $\mathscr{X}$ such that $\mathscr{F}_{\geq 1}$ is a good flag on $\mathscr{Y}$ over $p$.
Then
\begin{equation}\label{eqn: FE3 bound of C prime}
\frac{C'\left(\overline{M}|_Y,Y\right)}{\log(p)}\leq\varepsilon[K_Y:\QQ]I(M|_Y)
\end{equation}
and
\begin{equation}\label{eqn: FE3 bound of C}
\widetilde{C}\left(\overline{M}|_Y,A|_Y+\rho_0M|_Y,Y,p,\varepsilon\right)\leq\varepsilon[K_Y:\QQ]\left(2I(M|_Y)+I(A|_Y+\rho_0M|_Y)\right).
\end{equation}
Let $\bm{w}_{\mathscr{F}_{\sbullet}}$ be the valuation map attached to $\mathscr{F}_{\sbullet}$, and consider the convex body
\[
\widehat{\Delta}^{\mathscr{F}_{\sbullet}}_{\rm YM}\left(\overline{A};rY\right)\coloneqq \overline{\left(\bigcup_{m\in\ZZ_{\geq 1}}\frac{1}{m}\bm{w}_{\mathscr{F}_{\sbullet}}\left(\widehat{\Gamma}^{\rm ss}\left(m\overline{A};mrY\right)\setminus\{0\}\right)\right)}
\]
for each $r\in\QQ$ with $0\leq r\leq \rho_0$ (see Definition~\ref{defn:Yuan_Moriwaki_Newton_Okounkov}).

\begin{claim}\label{clm:FE3_slice}
One has
\begin{align*}
&\lim_{r\downarrow 0}\frac{\vol_{\RR^{\dim X+1}}\left(\widehat{\Delta}^{\mathscr{F}_{\sbullet}}_{\rm YM}\left(\overline{A}\right)\right)-\vol_{\RR^{\dim X+1}}\left(\widehat{\Delta}^{\mathscr{F}_{\sbullet}}_{\rm YM}\left(\overline{A};rY\right)\right)}{r} \\
&\qquad\qquad\qquad\qquad\qquad\qquad\qquad\qquad \geq\vol_{\RR^{\dim X}}\left(\widehat{\Delta}^{\mathscr{F}_{\geq 1}}_{CL(X|Y)}\left(\overline{A}(-\varepsilon)\right)\right).
\end{align*}
\end{claim}

\begin{proof}[Proof of Claim~\ref{clm:FE3_slice}]
Obviously, one has
\begin{align*}
&\lim_{r\downarrow 0}\frac{\vol_{\RR^{\dim X+1}}\left(\widehat{\Delta}^{\mathscr{F}_{\sbullet}}_{\rm YM}\left(\overline{A}\right)\right)-\vol_{\RR^{\dim X+1}}\left(\widehat{\Delta}^{\mathscr{F}_{\sbullet}}_{\rm YM}\left(\overline{A};rY\right)\right)}{r} \\
&\qquad\qquad\qquad\qquad\qquad\qquad \geq\vol_{\RR^{\dim X}}\left(\widehat{\Delta}^{\mathscr{F}_{\sbullet}}_{\rm YM}\left(\overline{A}\right)\cap\{w_1=-\ord_Y(A)\}\right),
\end{align*}
where the right-hand side denotes the Euclidean volume of the slice of $\widehat{\Delta}^{\mathscr{F}_{\sbullet}}_{\rm YM}\left(\overline{A}\right)$ by the hyperplane $\{w_1=-\ord_Y(A)\}$.
By Proposition~\ref{prop: Estimates II inclusions}, one has
\begin{align*}
\vol_{\RR^{\dim X}}\left(\widehat{\Delta}^{\mathscr{F}_{\sbullet}}_{\rm YM}\left(\overline{A}\right)\cap\{w_1=-\ord_Y(A)\}\right)\geq\vol_{\RR^{\dim X}}\left(\widehat{\Delta}^{\mathscr{F}_{\geq 1}}_{CL(X|Y)}\left(\overline{A}(-\varepsilon)\right)\right)
\end{align*}
as required (see also \cite[Appendix]{Lazarsfeld_Mustata08}).
\end{proof}

By taking $m\to\infty$ in Theorem~\ref{thm: general lower bound}, one has
\begin{align*}
&\frac{\avol\left(\overline{D};E\right)-\avol\left(\overline{D};E+rY\right)}{r} \\
&\quad \geq (\dim X+1)!\frac{\vol_{\RR^{\dim X+1}}\left(\widehat{\Delta}^{\mathscr{F}_{\sbullet}}_{\rm YM}\left(\overline{A}\right)\right)-\vol_{\RR^{\dim X+1}}\left(\widehat{\Delta}^{\mathscr{F}_{\sbullet}}_{\rm YM}\left(\overline{A};rY\right)\right)}{r}\log(p) \\
&\quad\qquad\qquad\qquad -\varepsilon[K_Y:\QQ](\dim X+1)!\left(2I(M|_Y)+I(A|_Y+\rho_0M|_Y)\right)
\end{align*}
for any $r\in\RR$ with $0<r\leq\rho_0$ (see Lemma~\ref{lem:FE3_bounded} and \eqref{eqn: FE3 bound of C}).
Hence, by taking $r\downarrow 0$,
\begin{align*}
&\lim_{r\downarrow 0}\frac{\avol\left(\overline{D};E\right)-\avol\left(\overline{D};E+rY\right)}{r} \\
&\qquad \geq (\dim X+1)!\vol_{\RR^{\dim X}}\left(\widehat{\Delta}^{\mathscr{F}_{\geq 1}}_{CL(X|Y)}\left(\overline{A}(-\varepsilon)\right)\right)\log(p) \\
&\quad\qquad\qquad\qquad -\varepsilon[K_Y:\QQ](\dim X+1)!\left(2I(M|_Y)+I(A|_Y+\rho_0M|_Y)\right) \\
&\qquad \geq (\dim X+1)\avolq{X|Y}\left(\overline{A}(-\varepsilon)\right) \\
&\quad\qquad\qquad\qquad -\varepsilon[K_Y:\QQ](\dim X+1)!\left(3I(M|_Y)+I(A|_Y+\rho_0M|_Y)\right)
\end{align*}
(see Claim~\ref{clm:FE3_slice}, Corollary~\ref{cor:YuanMoriwaki}, and \eqref{eqn: FE3 bound of C prime}), which leads to the required estimate as $\varepsilon\downarrow 0$.
\end{proof}

\begin{proof}[Proof of Theorem~A]
Assume $\left(\overline{D};E\right)\in\aDDiv_{\QQ,\QQ}(X)$.
By Corollary~\ref{cor:upper_bound}, Theorem~\ref{thm: general lower bound}, and \cite[Proposition~8.1]{IkomaRem}, one has
\begin{align*}
&(\dim X+1)\adeg\left(\left(\overline{M}|_{\pi_*^{-1}(Y)}\right)^{\cdot\dim X}\right) \\
&\qquad \leq\lim_{r\downarrow 0}\frac{\avol\left(\overline{D};E\right)-\avol\left(\overline{D};E+rY\right)}{r}\leq(\dim X+1)\left.\left\langle\left(\overline{D};E\right)^{\cdot\dim X}\right\rangle\right|_Y
\end{align*}
for any $(\pi,\overline{M})\in\widehat{\Theta}^{\rm rw}_Y\left(\overline{D};E\right)$.
Hence, by Remark~\ref{rem:arpin} \eqref{item: rw approximation arpin},
\[
\lim_{r\downarrow 0}\frac{\avol\left(\overline{D};E\right)-\avol\left(\overline{D};E+rY\right)}{r}=(\dim X+1)\left.\left\langle\left(\overline{D};E\right)^{\cdot\dim X}\right\rangle\right|_Y
\]
for every $Y$-big pair $\left(\overline{D};E\right)\in\aDDiv_{\QQ,\QQ}(X)$.
Consider a finite dimensional and $\QQ$-rational $\RR$-vector subspace of $\aDDiv_{\RR,\RR}(X)$ containing both $\left(\overline{D};E\right)$ and $(0;Y)$.
Since the function $\left(\overline{D}';E'\right)\mapsto\avol\left(\overline{D}';E'\right)^{1/(\dim X+1)}$ is concave (see \cite[Theorem~2.24(3)]{IkomaDiff1}) and the function
\[
\avol\left(\overline{D}';E'\right)^{-\dim X/(\dim X+1)}\cdot\left.\left\langle\left(\overline{D}';E'\right)^{\cdot\dim X}\right\rangle\right|_Y
\]
is locally Lipschitz-continuous (see Remark~\ref{rem:arpin} \eqref{item: remark arpin}), one obtains the theorem by Lemma~\ref{lem: fund property of concave function}.
\end{proof}

%%%
\section*{Acknowledgement}
\addcontentsline{toc}{section}{Acknowledgement}

The author is grateful to Professors Namikawa, Yoshikawa, and Moriwaki, and to Kyoto University for their supports.
This work was partially supported by Institut de Math\'ematiques de Jussieu -- Paris Rive Gauche, where part of this work was performed.
The author is grateful to Professors Merel and Chen for their warm welcome during the stay.

The author's work was supported by Japan Society for the Promotion of Science KAKENHI grant JP16K17559, JP20K03548, and partially JP16H06335.

%%%
\addcontentsline{toc}{section}{References}
\bibliography{ikoma}
\bibliographystyle{plain}

\end{document}